\documentclass[11pt]{amsart}

\usepackage{amsfonts,amssymb,amsmath,amsthm,amscd,enumerate}
\usepackage{latexsym}
\usepackage{euscript}
\usepackage{color}

\newtheorem{theorem}{Theorem}[section]

\newtheorem{lemma}[theorem]{Lemma}
\newtheorem{corollary}[theorem]{Corollary}
\newtheorem{proposition}[theorem]{Proposition}

\title{On Lie nilpotent associative algebras}

\author{Claud W. G. Dias Jr}

\address{Departamento de Matem\'atica, Universidade de Bras\'\i lia, 70910-900 Bras\'\i lia, DF, Brasil}

\email{}

\author{Alexei Krasilnikov}

\address{Departamento de Matem\'atica, Universidade de Bras\'\i lia, 70910-900 Bras\'\i lia, DF, Brasil}

\email{alexei@unb.br}

\date{}

\begin{document}

\maketitle

\begin{abstract}
Let $G$ be a group generated by a set $X$. It is well known and easy to check that
\[
[g_1, g_2, \dots ,g_n] = 1 \mbox{ for all } g_i \in G \qquad \iff \qquad [x_1, x_2, \dots , x_n] =1 \mbox{ for all } x_i \in X.
\]
Let $L$ be a Lie algebra generated by a set $X$. Then it is also well known and easy to check that
\[
[h_1, h_2, \dots , h_n] = 0 \mbox{ for all } h_i \in L \qquad \iff \qquad [x_1, x_2, \dots ,x_n] = 0 \mbox{ for all } x_i \in X.
\]

Now let $A$ be a unital associative algebra generated by a set $X$. Then the assertion similar to the above does not hold: for $n > 2$, it is easy to find an algebra $A$ with a generating set $X$ such that $[x_1, x_2, \dots ,x_n] = 0$ for all $x_i \in X$ but $[a_1, a_2, \dots ,a_n] \ne 0$ for some $a_i \in A$.

However, we prove the following result. Let $R$ be a unital associative and commutative ring such that $\frac{1}{3} \in R$. Let $A$ be a unital associative $R$-algebra generated by a set $X$. Let $X^2 = \{ x_1 x_2 \mid x_i \in X \}$ be the set of all products of $2$ elements of $X$. Then
\[
[a_1, a_2, \dots ,a_n] = 0 \mbox{ for all } a_i \in A \qquad \iff \qquad [y_1, y_2, \dots , y_n] =0 \mbox{ for all } y_i \in X \cup X^2.
\]
Moreover, one can assume that in the commutator $[y_1, y_2, \dots , y_n]$ above $y_1, y_n \in X$.
\end{abstract}



%%%%%%%%%%%%%%%%%%%%%%%%%%
%%%%%%%%%%%%%%%%%%%%%%%%%%
%%%%%%%%%%%%%%%%%%%%%%%%%%
%%%%%%%%%%%%%%%%%%%%%%%%%%
%%%%%%%%%%%%%%%%%%%%%%%%%%
\section{Introduction}

Let $G$ be a group generated by a set $X$. It is well known (see, for example, \cite[1.11 a), p. 258]{Huppert67} or \cite[Theorem 3.12]{Khukhro97}) and easy to check that
\begin{equation}
\label{assertiongroups}
[g_1, g_2, \dots ,g_n] = 1 \mbox{ for all } g_i \in G \qquad \iff \qquad [x_1, x_2, \dots , x_n] =1 \mbox{ for all } x_i \in X.
\end{equation}
Here a multiplicative commutator $[g_1, g_2, \dots ,g_n]$ is defined recursively by  $[g_1, g_2] = g_1^{-1} g_2^{-1} g_1 g_2$, $[g_1,  \dots , g_{n-1}, g_n] = \bigl[ [g_1,  \dots , g_{n-1}], g_n \bigr]$ $(n>2)$. 

Let $L$ be a Lie algebra generated by a set $X$. Then it is also well known (see, for instance, \cite[Corollary 5.20]{Khukhro97}) and easy to check that
\begin{equation}
\label{assertionLiealgebras}
[h_1, h_2, \dots , h_n] = 0 \mbox{ for all } h_i \in L \qquad \iff \qquad [x_1, x_2, \dots ,x_n] = 0 \mbox{ for all } x_i \in X.
\end{equation}
Here $[h_1, h_2]$ is the Lie product of $h_1$ and $h_2$ and the commutator $[h_1, h_2, \dots , h_n ] $ is defined recursively by $[h_1, \dots , h_{n-1}, h_n ]  = \bigl[ [h_1, \dots , h_{n-1}], h_n \bigr]$ $(n>2)$. 

Now let $A$ be a unital associative algebra generated by a set $X$. Then, for $A$, the assertion similar to (\ref{assertiongroups}) and (\ref{assertionLiealgebras}) does not hold: for $n > 2$, there are algebras $A$ generated by a set $X$ such that  $[x_1, x_2, \dots , x_n] = 0$ for all $x_i \in X$ but  $[a_1, a_2, \dots ,a_n] \ne 0$ for some $a_i \in A$. Indeed, suppose, for example, that $H$ is the $3$-dimensional Lie algebra with a basis $\{ a, b, c \}$ such that $[a,b] = c$, $[a,c] = [b,c] = 0$. Let $A = U(H)$ be the universal envelope of $H$. Then the set $X = \{ a, b \} \subset A$ generates the algebra $A$, we have $[x_1, x_2, x_3] = 0$ for all $x_i \in X$ but, for each $k \ge 1$,
\[
[a, \underbrace{ab, ab, \dots , ab}_{\mbox{$k$ times} } ] = a c^k \ne 0.
\]
Here an additive commutator $[a_1, a_2, \dots ,a_n]$ is defined recursively by  $[a_1, a_2] = a_1 a_2 - a_2 a_1$ and $[a_1,  \dots , a_{n-1}, a_n] = \bigl[ [a_1,  \dots , a_{n-1}], a_n \bigr]$ $(n>2)$. 

However, the following theorem holds. 

%%%%%%%%%%%%%%%%%%%%%%%%%%%%%%%
\begin{theorem}
\label{verification[a1an]=0}
Let $R$ be an arbitrary unital associative and commutative ring such that $3$ is invertible in $R$ (that is, $\frac{1}{3} \in R$). Let $A$ be a unital associative $R$-algebra generated by a set $X$ and let $n \ge 2$. Then we have
\[
[a_1, a_2, \dots , a_n] =0 \quad \mbox{for all} \quad a_i \in A
\]
if, and only if,
\[
[y_1, y_2, \dots , y_n] = 0 \quad \mbox{for all} \quad y_1, y_n \in X, \  y_2, \dots , y_{n-1} \in X \cup X^2 .
\]
\end{theorem}

Here $3 = 1 + 1 +1 \in R$ and $X^2 = \{ x_1 x_2 \mid x_i \in X \}$ is the set of all products of $2$ elements of $X$.

\medskip
Recall that an associative algebra $A$ is called \textit{Lie nilpotent} (of class at most $n-1$) if $[a_1, a_2, \dots , a_n] = 0$ for all $a_i \in A$. The study of Lie nilpotent associative rings and algebras was started by Jennings \cite{Jennings47} in 1947. Since then Lie nilpotent associative rings and algebras have been investigated in many papers from various points of view. Recent interest in Lie nilpotent associative algebras has been motivated by the study of the quotients $L_{n-1}/L_{n}$ of the lower central series of the associated Lie algebra of an associative algebra $A$; here $L_n$ is the linear span in $A$ of the set of all commutators $[a_1, a_2, \dots , a_n]$ $(a_i \in A)$. The study of these quotients $L_{n - 1} /L_{n}$ was initiated in 2007 in the pioneering article of Feigin and Shoikhet \cite{FS07} for $A = \mathbb C \langle X \rangle$; further results on this subject can be found, for example, in \cite{AJ10, BB11, BJ10, BEJKL12, CFZ15, Dangovski15, DK15, DK17, DE08, DKM08, EKM09, JO15, Kras13} and in the survey \cite{AE16}. 

\medskip
Let $R$ be a unital associative and commutative ring and let $A$ be a unital associative algebra over $R$ generated by a set $X$. Let $T^{(n)} = T^{(n)} (A)$ $(n \ge 2)$ be the two-sided ideal of $A$ generated by all commutators $[a_1, a_2, \dots , a_n]$ $(a_i \in A)$. Define $T^{(1)} (A) = A$. 

To study the Lie nilpotent quotients $A /T^{(n)}$ of an algebra $A$ one needs a ``good'' generating set for the two-sided ideal $T^{(n)}$ $(n \ge 2)$. Clearly, by definition, the ideal $T^{(n)}$ is generated by the set of all commutators $[a_1, a_2, \dots , a_n]$ $(a_i \in A )$. However, this generating set is very large and inconvenient to use; in particular, this set is infinite even if $X$ is finite. For $n \le 5$, smaller ``good'' generating sets of $T^{(n)}$ have been found and used to study the quotients $A /T^{(n)}$. 

It can be easily seen that the ideal $T^{(2)}$ is generated (as a two-sided ideal in $A$) by the commutators $[x_1, x_2]$ $(x_i \in X)$. The ideal $T^{(3)}$ is generated by the polynomials
\[
[x_1,x_2,x_3], \qquad [x_1,x_2][x_3,x_4] + [x_1,x_3][x_2,x_4] \qquad (x_i \in X)
\]
(\cite{Latyshev63}; see also, for instance, \cite{BEJKL12, GK01, GK99}). If $\frac{1}{3} \in R$ then a similar generating set for $T^{(4)}$ contains the polynomials of 3 types (\cite{Volichenko78}; see also \cite{EKM09, Gordienko07}):
\begin{equation}
\label{polynomial1}
[x_1,x_2,x_3,x_4] \qquad (x_i \in X),
\end{equation}
\begin{equation}
\label{polynomial2}
[x_1,x_2][x_3,x_4,x_5] \qquad (x_i \in X),
\end{equation}
\begin{equation}
\label{polynomial3}
\bigl( [x_1,x_2] [x_3,x_4] + [x_1,x_3] [x_2,x_4] \bigr) [x_5,x_6] \qquad (x_i \in X).
\end{equation}
If $\frac{1}{3} \notin R$ then a set of generators of $T^{(4)}$ consists of the polynomials of types (\ref{polynomial1}), (\ref{polynomial3}) and the polynomials
\begin{align*}
&   [x_1,x_2,x_3] [x_4,x_5,x_6] \qquad (x_i \in X),
\\
& [x_1,x_2,x_3] [x_4,x_5] + [x_1,x_2,x_4][x_3,x_5] \qquad (x_i \in X),
\\
& [x_1,x_2,x_3] [x_4,x_5] + [x_1,x_4,x_3][x_2,x_5] \qquad (x_i \in X)
\end{align*}
(see \cite[Theorem 1.3]{DK15}). In the latter case, in general, the polynomials $[x_1,x_2][x_3,x_4,x_5]$ $(x_i \in X)$ of type (\ref{polynomial2}) do not belong to $T^{(4)}$ but $3 \, [x_1,x_2][x_3,x_4,x_5] \in T^{(4)}$ (see \cite{Kras13}). For arbitrary $R$, a set of generators of $T^{(5)}$ consists of polynomials of $8$ types, see \cite{dCKras13}.

The following theorem provides a simple `small' generating set for the ideal $T^{(n)}$ for arbitrary $n \ge 2$. It is clear that this theorem is equivalent to Theorem \ref{verification[a1an]=0}.

%%%%%%%%%%%%%%%%%%%%%%%%%%%%%%%%%%%%%
\begin{theorem}
\label{generatorsTn}
Let $R$ be an arbitrary unital associative and commutative ring such that $\frac{1}{3} \in R$. Let $A$ be a unital associative $R$-algebra generated by a set $X$. Then the ideal $T^{(n)}$ $(n \ge 2)$ is generated (as a two-sided ideal in $A$) by the set
\[
\bigl\{ [y_1, y_2, \dots , y_n] \mid y_1, y_n \in X; \  y_2, \dots, , y_{n-1} \in X \cup X^2  \bigr\} .
\]
\end{theorem}

We will prove Theorem \ref{generatorsTn} by induction on $n$: assuming that the theorem is true for $T^{(n-2)}$ we will show that  it is valid for $T^{(n)}$. The cases $n = 1$ and $n = 2$ form the base of the induction. 

The following proposition makes the induction step possible; it has been proved in a different slightly weaker form by Grishin and Pchelintsev \cite[Lemma 2]{GP15}.

%%%%%%%%%%%%%%%%%%%%%%%%%%
\begin{proposition}[cf. \cite{GP15}]
\label{A[Tn-2AA]A=Tn}
Let $R$ be an arbitrary unital associative and commutative ring such that $\frac{1}{3} \in R$. Let $A$ be a unital associative algebra over $R$. Let $n \in \mathbb Z$, $n > 2$. Then the two-sided ideal of $A$ generated by the set $[T^{(n-2)}, A, A]$ coincides with $T^{(n)}$, that is,
\[
A [T^{(n-2)} , A, A]A  = T^{(n)} .
\]
\end{proposition}

We will prove Proposition \ref{A[Tn-2AA]A=Tn} in the Preliminaries section. 

%%%%%%%%%%%%%%%%%%%%%%%%%%%
\medskip
\noindent
\textbf{Remarks.} 
1. In general, $A[T^{(n-1)}, A] A \ne T^{(n)}$; for details,  see the remarks at the end of this section.

2. In general, Proposition \ref{A[Tn-2AA]A=Tn}  fails if $\frac{1}{3} \notin R$. In particular, it fails over $\mathbb Z$ and over a field of characteristic 3;  for details,  see the remarks at the end of this section.

\medskip
To prove Theorem \ref{generatorsTn} we make use of the following theorem that is also of independent interest.

%%%%%%%%%%%%%%%%%%%%%%%%%%
\begin{theorem}
\label{generatorsW}
Let $R$ be an arbitrary unital associative and commutative ring and let $A$ be a unital associative $R$-algebra generated by a  set $X$. Let $U$ be a two-sided ideal in $A$ generated by a set $S$ and let $W$ be the two-sided ideal in  $A$ generated by the set
\[
\bigl\{ [u, a_1, a_2] \mid u \in U, \ a_1, a_2 \in A \bigr\} .
\]
Then $W$ is generated (as a two-sided ideal in $A$) by the following elements:
\begin{equation}
\label{generator[sxx]}
[ s, x_1, x_2] \qquad (s \in S, \ x_i \in X ),
\end{equation}
\begin{equation}
\label{generators[xxx]}
s [ x_1, x_2, x_3] \qquad (s \in S, \ x_i \in X ),
\end{equation}
\begin{equation}
\label{generator[sx][xxx]}
[ s, x_1] [x_2,x_3,x_4] \qquad (s \in S, \ x_i \in X ),
\end{equation}
\begin{equation}
\label{generator[sx][xx]}
[ s, x_1] [x_2, x_3] + [s, x_2][x_1, x_3] \qquad (s \in S, \ x_i \in X ),
\end{equation}
\begin{equation}
\label{generators[xx][xx]}
s \bigl( [x_1, x_2] [x_3, x_4] + [x_1, x_3] [x_2, x_4] \bigr)  \qquad (s \in S, \ x_i \in X ).
\end{equation}
\end{theorem}

%%%%%%%%%%%%%%%%%%%%%%%%%%%%%%%%%%%
Let $I'$ be the two-sided ideal of $A$ generated by all elements of the forms (\ref{generator[sxx]}) and   (\ref{generator[sx][xx]}). We will prove the following proposition in Section 3.

\begin{proposition}
\label{proposition3vinI'}
Under the hypotheses of Theorem $\ref{generatorsW}$, for each element $v =  [s, x_1] [x_2,x_3,x_4]$ of the form $(\ref{generator[sx][xxx]})$, we have $3 \, v \in I'$. In particular, if $3$ is invertible in $R$ then $v \in I'$.
\end{proposition}

%%%%%%%%%%%%%%%%%%%%%%%%%%%%%%%%%%%
The assertion below follows immediately from Theorem \ref{generatorsW} and Proposition \ref{proposition3vinI'}.

%%%%%%%%%%%%%%%%%%%%%%%%%%%%%%%%%%%
\begin{corollary}
\label{corollarygeneratorsW}
If $\frac{1}{3} \in R$ then, under the hypotheses of Theorem $\ref{generatorsW}$, the ideal $W$ of $A$ is generated by the elements of the forms $(\ref{generator[sxx]})$--$(\ref {generators[xxx]})$ and  $(\ref{generator[sx][xx]})$--$(\ref{generators[xx][xx]})$.
\end{corollary}

%%%%%%%%%%%%%%%%%%%%%%%%%%%%%%%%%%%
\noindent
\textbf{Remarks.}
1. In general, $A[T^{(n-1)}, A] A \ne T^{(n)}$ because, in general, $[T^{(n-1)}, A] \nsubseteq T^{(n)}$. More precisely, if $\frac{1}{6} \in R$ and $n = 2k + 1$ is odd then $[T^{(2k + 1)}, A] \subseteq T^{(2k + 2)}$, therefore $A[T^{(2k + 1)}, A] A = T^{(2k + 2)}$ $(k \ge 0)$. However, in general,  if $n = 2k$ is even then $[T^{(2k )}, A] \nsubseteq T^{(2k + 1)}$ so  $A [T^{(2k )}, A] A \ne T^{(2k + 1)}$ $(k \ge 1)$. 

Indeed, one can check that each element of $T^{(n-1)}$ is a sum of elements of the form $[a_1, \dots , a_{n-1}] b$ where $a_i, b \in A$. Hence, each element of $[T^{(n-1)}, A]$ is a sum of elements of the form 
\[
\bigl[ [a_1, \dots , a_{n-1}] b, c \bigr] = [a_1, \dots , a_{n-1} , c] b + [a_1, \dots , a_{n-1}] [b, c]  \qquad (a_i, b, c \in A).
\]
Since $[a_1, \dots , a_{n-1} , c] b \in T^{(n)}$, we have $\bigl[ [a_1, \dots , a_{n-1}] b, c \bigr] \in T^{(n)}$ if, and only if, $[a_1, \dots , a_{n-1}] [b, c] \in T^{(n)}$. It follows that $[T^{(n-1)}, A] \subseteq T^{(n)}$ if, and only if, $T^{(n-1)} T^{(2)} \subseteq T^{(n)}$.

Now it remains to note that if $\frac{1}{6} \in R$ then, for $k \ge 0$, we have $T^{(2k + 1)} T^{(2)} \subseteq T^{(2k + 2)}$ (see \cite{BJ10, GP15, SS90}). On the other hand, in general, for $k \ge 1$, we have $T^{(2k)} T^{(2)} \nsubseteq T^{(2k+1)}$ (\cite{GTS11}, see also \cite{DK17,GP15}).

2. In general, $[T^{(n-2)} , A, A]  \nsubseteq  T^{(n)}$ if $\frac{1}{3} \notin R$. In particular, if $R = \mathbb Z$ and $A = \mathbb Z \langle X \rangle$ is the free unital associative ring on a free generating set $X$, $|X| \ge 5$, then $[ T^{(2)}, A, A] \nsubseteq T^{(4)}$.

Indeed, we have
\[
\bigl[ [x_1, x_2] x_3, x_4, x_5 \bigr] = [x_1, x_2] [x_3, x_4, x_5] + [x_1, x_2, x_4] [x_3, x_5] + [x_1, x_2, x_5] [x_3, x_4] + [x_1, x_2, x_4, x_5] x_3 .
\]
It is clear that $[x_1, x_2, x_4, x_5] x_3 \in T^{(4)}$. On the other hand, by Lemma \ref{lemma[ua][aa]3u[aaa]}  below, $ [x_1, x_2, x_4] [x_3, x_5] + [x_1, x_2, x_5] [x_3, x_4] \in T^{(4)}$. It follows that $\bigl[ [x_1, x_2] x_3, x_4, x_5 \bigr] \in T^{(4)}$ if, and only if, $[x_1, x_2] [x_3, x_4, x_5]  \in T^{(4)}$. However, if $x_1,  \dots , x_5$ are distinct elements of $X \subset \mathbb Z \langle X \rangle$ then, by \cite{Kras13}, $[x_1, x_2] [x_3, x_4, x_5]  \notin T^{(4)}$ (although $3 \ [x_1, x_2] [x_3, x_4, x_5]  \in T^{(4)}$). Thus, for $A = \mathbb Z \langle X \rangle$, $|X| \ge 5$, we have  $[ T^{(2)}, A, A] \nsubseteq T^{(4)}$, as claimed.

3. Let $A$ be a unital associative algebra generated by a set $X$ and let $U$ be a two-sided ideal in $A$ generated by a set $S$. It is easy to check that the two-sided ideal $V$ of $A$ generated by the set
\[
\{ [u, a] \mid u \in U, a \in A \}
\]
can be generated by the elements $[s,x]$  and $s [x_1, x_2]$ $( s \in S, x, x_i \in X)$. 

4. The set
\[
S^{(n)} = \bigl\{ [y_1, y_2, \dots , y_n] \mid y_1, y_n \in X, \ y_2, \dots , y_{n-1} \in X \cup X^2 \bigr\} 
\]
of generators of the two-sided ideal $T^{(n)}$ of $A$ described in Theorem \ref{generatorsTn} is not, in general, a minimal generating set for $T^{(n)}$. Indeed, let, for instance,
\begin{align*}
S = & \bigl\{ [y_1, y_2, \dots , y_n] \mid   y_k \in X^2 \mbox{ for some } k, 2 \le k \le n-1; \ \ y_i \in X \mbox{ for all } i \ne k \bigr\} ,
\\
S' = & \bigl\{ [y_1, y_2, \dots , y_n] \mid  y_2 \in X^2, \ y_i \in X \mbox{ for all } i \ne 2 \bigr\} .
\end{align*}
If $n > 3$ then $S' \subsetneqq S \subsetneqq S^{(n)}$; on the other hand, one can easily check that $S$ is contained in the $\mathbb Z$-linear span of $S'$ so the sets $S^{(n)}$ and $(S^{(n)} \setminus S) \cup S'$ generate the same ideal $T^{(n)}$.   

5. Theorem \ref{verification[a1an]=0} shows that if $\frac{1}{3} \in R$ then an associative $R$-algebra $A$ generated by a set $X$ is Lie nilpotent if and only if its Lie subalgebra generated by the set $X \cup X^2$ is nilpotent. In particular, if $X$ is finite and every finitely generated Lie subalgebra of $A$ is nilpotent (of some class depending on the Lie subalgebra) then $A$ is Lie nilpotent. This answers (for an $R$-algebra $A$ with $\frac{1}{3} \in R$) a question  asked by Sysak \cite[Question 3.1]{Sysak10}.

6. Bapat and Jordan \cite[Corollary 1.5]{BJ10} have used of a result of Arbesfeld and Jordan \cite[Theorem 1.3]{AJ10} to  prove the following. Let $F$ be a field of characteristic $0$ and let $A$ be an associative $F$-algebra generated by a set $X$. Define $L_1 = A$, $L_{n+1} = [L_n, A]$ $(n=1,2, \dots )$. Then $L_n / L_{n+1}$ is spanned by the commutators $[z, y_1, y_2, \dots y_{n-1}] + L_{n+1}$ where $z$ is a product of (any number of) elements of $X$, $y_1 \in X$ and $y_i \in X \cup X^2$ $(i = 2, \dots , n-1)$. This result and Theorem \ref{generatorsTn} have `similar flavor'. However, to the best of our knowledge, these results are independent in a sense that none of them can be deduced from another.

%%%%%%%%%%%%%%%%%%%%%%%%%%%%%%%%%%%
%%%%%%%%%%%%%%%%%%%%%%%%%%%%%%%%%%%
%%%%%%%%%%%%%%%%%%%%%%%%%%%%%%%%%%%
%%%%%%%%%%%%%%%%%%%%%%%%%%%%%%%%%%%
%%%%%%%%%%%%%%%%%%%%%%%%%%%%%%%%%%%

\section{Preliminaries}

Let $A$ be an associative ring and let $a_i, b_j \in A$. In the proofs we will make use of the following identities:
\[
[a_1 a_2, b] = a_1 [a_2, b] + [a_1, b] a_2
\]
and, more generally,
\[
[a_1 \dots a_k, b] = \sum_{i=1}^k a_1 \dots a_{i-1} [a_i, b] a_{i+1} \dots a_k ;
\]
\[
[a_1 a_2, b_1, b_2] = a_1 [a_2, b_1, b_2] + [a_1, b_1][a_2, b_2] + [a_1, b_2][a_2, b_1] + [a_1, b_1, b_2] a_2
\]
and, more generally,
\begin{align*}
[a_1 \dots a_k, b_1, b_2] = & \sum_{i=1}^k a_1 \dots a_{i-1} [a_i, b_1, b_2] a_{i+1} \dots a_k 
\\
+ & \sum_{1 \le i < i' \le k} \bigl( a_1 \dots a_{i-1} [a_i, b_1] a_{i+1} \dots a_{i' - 1} [a_{i'}, b_2] a_{i' + 1} \dots a_k 
\\
+ & a_1 \dots a_{i-1} [a_i, b_2] a_{i+1} \dots a_{i' - 1} [a_{i'}, b_1] a_{i' + 1} \dots a_k \bigr) ;
\end{align*}
\begin{align*}
[a_1 a_2, b_1, b_2, b_3] = & a_1 [a_2, b_1, b_2, b_3] + [a_1, b_1] [a_2, b_2, b_3] + [a_1, b_2][a_2, b_1, b_3] + [a_1, b_3][a_2, b_1, b_2] 
\\
+ & [a_1, b_1, b_2] [a_2, b_3] + [a_1, b_1, b_3] [a_2, b_2] + [a_1, b_2, b_3] [a_2, b_1] + [a_1, b_1, b_2, b_3] a_2 .
\end{align*}

The following lemma is well known.

%%%%%%%%%%%%%%%%%%%%%%%%%%%%%%%%
\begin{lemma}
\label{L2}
Let $A$ be an associative $R$-algebra generated by a set $X$. Then, for all $a,b \in A$, the commutator $[a,b]$ is a linear combination of the commutators of the form $[d, x]$ where $x \in X$, $d \in A$.
\end{lemma}

%%%%%%%%%%%%%%%%%%%%%%%%%%%%%%%%
\begin{proof}
It is clear that each commutator $[a,b]$ $(a, b \in A)$ is a linear combination over $R$ of the commutators of the form $[y_1 \dots y_k , z_1 \dots z_{\ell}]$ where $k, \ell \ge 1$, $y_i, z_j \in X$. On the other hand,
\begin{align*}
& [y_1 \dots y_k , z_1 \dots z_{\ell}] = y_1 \dots y_k  z_1 \dots z_{\ell} -  z_1 \dots z_{\ell} y_1 \dots y_k
\\
= \ & \bigl( y_1 \dots y_k  z_1 \dots  z_{\ell} - z_{\ell} y_1 \dots y_k  z_1 \dots  z_{\ell - 1}  \bigr) + \bigl( z_{\ell} y_1 \dots y_k  z_1 \dots  z_{\ell - 1}  - z_{\ell - 1} z_{\ell} y_1 \dots y_k  z_1 \dots z_{\ell - 2}  \bigr)
\\
+ \ & \dots + \bigl(  z_2 \dots z_{\ell} y_1 \dots y_{k} z_1  -  z_1 \dots z_{\ell} y_1 \dots y_k \bigr)
\\
= \ & [y_1 \dots y_k z_1 \dots z_{\ell - 1}, z_{\ell}] + [z_{\ell} y_1 \dots y_k  z_1 \dots z_{\ell -2}, z_{\ell - 1} ] + \dots + [ z_2 \dots z_{\ell} y_1 \dots y_{k}, z_1] .
\end{align*}
The result follows.
\end{proof}

We need the following lemmas.

%%%%%%%%%%%%%%%%%%%%%%%%%%%%%%%%%
\begin{lemma}
\label{lemma[dz][zz]=0}
Let $B$ be an associative ring. Let $d \in B$, $Z \subset B$. Suppose that $[d, z_1, z_2] = 0$ and $[d, z_1 z_2, z_3] = 0$ for all $z_i \in Z$. Then, for all $z_i \in Z$, 
\[
[d, z_1] [z_2, z_3] + [d, z_2] [z_1, z_3] = 0 .
\]
\end{lemma}

%%%%%%%%%%%%%%%%%%%%%%%%%%%%%%%%%
\begin{proof}
Let $z_1, z_2, z_3 \in Z$ be arbitrary elements of $Z$. We have $[d, z_1 z_2, z_3] =0$. On the other hand, 
\begin{align*}
& [d, z_1 z_2, z_3] = - [z_1 z_2, d, z_3] = - z_1 [z_2, d, z_3] - [z_1, d] [z_2, z_3] - [z_1, z_3] [z_2, d] - [z_1, d, z_3] z_2
\\
& = z_1 [d, z_2, z_3] + [d, z_1] [z_2, z_3]  + [d, z_2] [z_1, z_3] - \bigl[ [d, z_2], [z_1, z_3] \bigr] + [d, z_1, z_3] z_2 .
\end{align*}
Since $[d, z_j, z_k] = 0$ and $\bigl[ [d, z_2], [z_1, z_3] \bigr]  = [d, z_2, z_1, z_3] - [d, z_2, z_3, z_1] = 0$ for all $z_i \in Z$, we have $[d, z_1] [z_2, z_3]  + [d, z_2] [z_1, z_3] = 0$, as required.
\end{proof}

%%%%%%%%%%%%%%%%%%%%%%%%%%%%%%%%%
\begin{lemma}
\label{lemma[cz][zzz]=0}
Let $B$ be an associative ring. Let $c \in B$, $Z \subset B$. Suppose that $[c, z_1, z_2, z_3] = 0$, $[c, z_1, z_2 z_3, z_4] = 0$ and $[c, z_1 z_2, z_3, z_4] = 0$ for all $z_i \in Z$. Then, for all $z_i \in Z$,
\[
[c, z_1] [z_2, z_3, z_4] + [c, z_2] [z_1, z_3, z_4] = 0 .
\]
\end{lemma}

%%%%%%%%%%%%%%%%%%%%%%%%%%%%%%%%%
\begin{proof}
Let $z_1 \in Z$ be an arbitrary element of $Z$. Let $d = [c, z_1]$. Then we have $[d, z_2, z_3] = [d, z_2 z_3, z_4] = 0$ for all $z_i \in Z$ so, by Lemma \ref{lemma[dz][zz]=0}, $[d, z_2] [z_3,z_4] + [d, z_3] [z_2, z_4] = 0$ for all $z_i \in Z$. In other words, for all $z_i \in Z$, we have
\begin{equation}
\label{[czz][zz]=0}
[c, z_1, z_2][z_3, z_4] + [c, z_1, z_3] [z_2, z_4] = 0 .
\end{equation}

Further, let $z_1, z_2, z_3, z_4 \in Z$ be arbitrary elements of $Z$. We have $[c, z_1 z_2, z_3, z_4] = 0$. On the other hand,
\begin{align*}
& [c, z_1 z_2, z_3, z_4] = - [z_1 z_2, c, z_3, z_4] =  - z_1 [z_2, c, z_3, z_4] - [z_1, c] [z_2, z_3, z_4] - [z_1, z_3] [z_2, c, z_4] 
\\
& - [z_1 , z_4] [z_2, c, z_3] - [z_1, c, z_3][ z_2, z_4] - [z_1, c, z_4] [z_2, z_3] - [z_1, z_3, z_4] [z_2, c] - [z_1, c, z_3, z_4] z_2
\\
& =  z_1 [c, z_2, z_3, z_4]  + \bigl(  [c, z_1] [z_2, z_3, z_4] + [c, z_2] [z_1, z_3, z_4] \bigr) -  \bigl[ [c, z_2], [z_1, z_3, z_4] \bigr]  - \bigl(  [c, z_2, z_4] [z_3, z_1] 
\\
& +  [c, z_2, z_3]  [z_4, z_1] \bigr) +  \bigl[ [c, z_2, z_4], [z_3, z_1] \bigr] + \bigl[ [c, z_2, z_3] , [z_4, z_1] \bigr] - \bigl( [c, z_1, z_3][ z_4, z_2] + [c, z_1, z_4] [z_3, z_2] \bigr)
\\
& + [c, z_1, z_3, z_4] z_2 .
\end{align*}
Clearly, $ [c, z_i, z_3, z_4] = 0$ $(i = 1, 2)$. By (\ref{[czz][zz]=0}), we have $ [c, z_i, z_4] [z_3, z_j] 
+  [c, z_i, z_3]  [z_4, z_j] = 0$ $(i,j = 1,2)$. It is clear that the commutators $\bigl[ [c, z_2], [z_1, z_3, z_4] \bigr]$, $\bigl[ [c, z_2, z_4], [z_3, z_1] \bigr] $ and $\bigl[ [c, z_2, z_3] , [z_4, z_1] \bigr]$  are linear combinations of the commutators of the form $[c, z_{i_1}, z_{i_2} , z_{i_3}, z_{i_4}] $ so these commutators are equal to $0$. It follows that 
\[
[c, z_1] [z_2, z_3, z_4] + [c, z_2] [z_1, z_3, z_4] = 0, 
\]
as required.
\end{proof}

The following lemma and its corollary are well known (see, for instance, \cite[Lemma 2.3 and Corollary 2.4]{Kras13}).

%%%%%%%%%%%%%%%%%%%%%%%%%%%%%%%%%
\begin{lemma}
\label{sigma[sz][zzz]}
Let $B$ be an associative ring. Let $c \in B$, $Z \subset B$. Suppose that
\[
[c, z_1] [z_2, z_3, z_4] + [c, z_2] [z_1, z_3, z_4] = 0
\]
for all $z_i \in Z$. Then, for all $z_i \in Z$ and all permutations $\sigma$ of the set $\{ 1, 2, 3, 4 \}$, 
\[
[c, z_{\sigma (1) }] [z_{\sigma (2) }, z_{\sigma (3) }, z_{\sigma (4) } ] = (- 1)^{\sigma} [c, z_1] [z_2, z_3, z_4].
\]
\end{lemma} 

\begin{proof}
We have 
\begin{equation}
\label{[s2][134]}
[c, z_1] [z_2, z_3, z_4] = - [c, z_2] [z_1, z_3, z_4] .
\end{equation}
On the other hand, by the Jacobi identity and (\ref{[s2][134]}), 
\begin{align*}
& [c, z_1] [z_2, z_3, z_4] = - [c, z_1] \bigl[ z_4, [z_2, z_3] \bigr] =  - [c, z_1] [z_4, z_2, z_3] + [c, z_1] [z_4, z_3, z_2] 
\\
=  \ & [c, z_4] [z_1, z_2, z_3] - [c, z_4] [z_1, z_3, z_2] = [c, z_4] \bigl[ z_1, [z_2, z_3] \bigr] = - [c, z_4] [z_2, z_3, z_1] ,
\end{align*}
that is, 
\begin{equation}
\label{[s4][231]}
 [c, z_1] [z_2, z_3, z_4]  = - [c, z_4] [z_2, z_3, z_1] .
\end{equation}
Further, we have 
\begin{equation}
\label{[s1][324]}
 [c, z_1] [z_2, z_3, z_4]  = -  [c, z_1] [z_3, z_2, z_4] .
\end{equation}
It follows from (\ref{[s2][134]}), (\ref{[s4][231]}) and (\ref{[s1][324]}) that 
\[
 [s, z_{\tau (1)}] [z_{\tau (2)}, z_{\tau (3)}, z_{\tau (4)}]  = (- 1)^{\tau}  [s, z_1] [z_2, z_3, z_4] 
\]
where $\tau$ is any of the transpositions $(12)$, $(14)$, $(23)$. Suppose that  $\sigma = \tau_1 \tau_2 \dots \tau_k$ where $k >1$, $\tau_i \in \{ (12), (14), (23) \}$ for all $i$;  then
\[
[c, z_{\sigma (1) }] [z_{\sigma (2) }, z_{\sigma (3) }, z_{\sigma (4) } ] = (- 1)^{\sigma} [c, z_1] [z_2, z_3, z_4].
\]
Since the transpositions $(12)$, $(14)$, $(23)$ generate the entire group $S_4$ of the permutations of the set $\{ 1, 2, 3, 4 \}$, the result follows.
\end{proof}

%%%%%%%%%%%%%%%%%%%%%%%%%%%%%%%%%%%%
\begin{corollary}
\label{corollary3[sz][zzz]}
Under the hypotheses of Lemma $\ref{sigma[sz][zzz]}$,  for all $z_i \in Z$, we have
\[
3 \ [c, z_1] [z_2, z_3, z_4] = 0. 
\]
\end{corollary}

\begin{proof}
By the Jacobi identity, we have
\[
[c, z_1] [z_2, z_3, z_4] + [c, z_1] [z_3, z_4, z_2] + [c, z_1] [z_4, z_2, z_3] =  [c, z_1] \bigl( [z_2, z_3, z_4] + [z_3, z_4, z_2] + [z_4, z_2, z_3] \bigr) = 0.
\]
On the other hand, by Lemma \ref{sigma[sz][zzz]}, 
\[
[c, z_1] [z_2, z_3, z_4] = [c, z_1] [z_3, z_4, z_2] = [c, z_1] [z_4, z_2, z_3]. 
\]
The result follows.
\end{proof}

%%%%%%%%%%%%%%%%%%%%%%%%%%%%%%%%%%%
%%%%%%%%%%%%%%%%%%%%%%%%%%%%%%%%%%%
\begin{proof}[Proof of Proposition \ref{A[Tn-2AA]A=Tn} ]

This proposition has been proved by Grishin and Pchelintsev \cite[Lemma 2]{GP15} under the assumption that $6$ is invertible in $R$. Our proof is slightly different from one given in  \cite{GP15}, it is valid if $3$ is invertible in $R$. .

Since the ideal $T^{(n)}$ is generated by the commutators $[a_1, a_2, \dots , a_n]$ $(a_i \in A)$ and $[a_1, a_2, \dots , a_n] \in [T^{(n-2)}, A, A]$ for all  $a_i \in A$, we have 
\[
T^{(n)} \subseteq A[T^{(n-2)}, A, A] A.
\]
To prove the proposition it suffices to check that 
\[
[T^{(n-2)}, A, A] \subset T^{(n)};
\]
then $A[T^{(n-2)}, A, A]A \subseteq T^{(n)}$ and therefore $A[T^{(n-2)}, A, A]A = T^{(n)}$.

The following lemma is well known.

%%%%%%%%%%%%%%%%%%%%%%%%%%%%%%%%%%%%
\begin{lemma}
\label{lemma[ua][aa]3u[aaa]}
Let $R$ be an arbitrary unital associative and commutative ring and let $A$ be an associative $R$-algebra. Let $u = [b_1, b_2, \dots, b_{n-2}]$ $(b_i \in A)$. Let $a_i \in A$. Then
\begin{itemize}
\item[i)] $[u, a_1] [a_2, a_3] + [u, a_2] [a_1, a_3] \in T^{(n)} $;
\item[ii)] $3 \ u [a_1, a_2, a_3] \in T^{(n)}$.
\end{itemize}
\end{lemma}

%%%%%%%%%%%%%%%%%%%%%%%%%%%%%%%%%%%%
\begin{proof}
To prove i) we make use of Lemma \ref{lemma[dz][zz]=0} with $B=A/T^{(n)}$, $Z = B$. For all $a_1, a_2, a_3 \in A$, we have $[u, a_1, a_2] \equiv 0 \pmod{T^{(n)} }$ and $[u, a_1 a_2, a_3] \equiv 0 \pmod{ T^{(n)} }$. Then, by Lemma \ref{lemma[dz][zz]=0}, 
\[
[u, a_1] [a_2, a_3] + [u, a_2] [a_1, a_3]  \equiv 0 \pmod{T^{(n)} }
\]
for all $a_i \in A$, as required.

To prove ii) we apply Lemma \ref{lemma[cz][zzz]=0} and Corollary \ref{corollary3[sz][zzz]} with $B = A/T^{(n)}$ and $Z = A$. Let $v = [b_1, b_2, \dots , b_{n-3}]$; then $u = [v, b_{n-2}]$. For all $a_i' \in A$, we have 
\[
[v, a_1', a_2', a_3'] \equiv 0 \pmod{T^{(n)} }, \ \  [v, a_1', a_2' a_3', a_4'] \equiv 0 \pmod{T^{(n)} },  \ \  [v, a_1' a_2', a_3', a_4'] \equiv 0 \pmod{T^{(n)} }.
\]
Then, by Lemma \ref{lemma[cz][zzz]=0}, we have $[v, a_1'] [a_2', a_3', a_4'] + [v, a_2'] [a_1', a_3', a_4'] \equiv 0 \pmod{T^{(n)} }$ for all $a_i' \in A$ and, by Corollary \ref{corollary3[sz][zzz]}, $3 \ [v, a_1'] [a_2', a_3', a_4'] \equiv 0 \pmod{T^{(n)} }$ for all $a_i' \in A$. In particular, $3 \ [v, b_{n-2}] [a_2', a_3', a_4'] \equiv 0 \pmod{T^{(n)} }$, that is, 
\[
3 \ u [a_1, a_2, a_3] \equiv 0 \pmod{T^{(n)} }
\]
for all $a_i \in A$, as required. 
\end{proof}

Now we are in a position to complete the proof of Proposition \ref{A[Tn-2AA]A=Tn}. Recall that it remains to check that $[T^{(n-2)}, A, A] \subset T^{(n)}$. It is clear that $T^{(n-2)}$ is spanned by the elements $a_1[b_1, b_2, \dots , b_{n-2}] a_2$ $(a_i, b_j \in A)$. Since $a_1[b_1, b_2, \dots , b_{n-2}] a_2 = [b_1, b_2, \dots , b_{n-2}] a_1 a_2  + \bigl[ [b_1, b_2], \dots , b_{n-2}, a_1 \bigr] a_2$, the ideal $T^{(n-2)}$ is also spanned by the elements $[b_1, b_2, \dots , b_{n-2}] a$ $(a, b_j \in A)$. It follows that the $R$-module $[T^{(n-2)}, A, A] $ is spanned over $R$ by the elements $[u a_1, a_2, a_3] $ where $u = [b_1, b_2, \dots , b_{n-2}]$, $a_i, b_j \in A$.

We have
\[
[u a_1, a_2, a_3] = u [a_1, a_2, a_3] + [u, a_2][a_1, a_3] + [u, a_3][a_1, a_2] + [u, a_2, a_3] a_1 .
\]
It is clear that $ [u, a_2, a_3] \in T^{(n)}$; by Lemma \ref{lemma[ua][aa]3u[aaa]} ii), $u [a_1, a_2, a_3] \in T^{(n)}$ (recall that $3$ is invertible in $R$) and, by  Lemma \ref{lemma[ua][aa]3u[aaa]} i), we have
\[
[u, a_2][a_1, a_3] + [u, a_3][a_1, a_2] = - \bigl( [u, a_2] [a_3, a_1] + [u, a_3] [a_2, a_1] \bigr) \in T^{(n)}.
\]
Hence, for all $u = [b_1, b_2, \dots , b_{n-2}]$ $(b_i \in A)$ and all $a_i \in A$, we have $[u a_1, a_2, a_3] \in T^{(n)}$. It follows that $[T^{(n-2)}, A, A] \subset T^{(n)}$, as required.

The proof of Proposition \ref{A[Tn-2AA]A=Tn}  is completed.
\end{proof}

%%%%%%%%%%%%%%%%%%%%%%%%%%%%%%%%%%%
%%%%%%%%%%%%%%%%%%%%%%%%%%%%%%%%%%%
%%%%%%%%%%%%%%%%%%%%%%%%%%%%%%%%%%%
%%%%%%%%%%%%%%%%%%%%%%%%%%%%%%%%%%%
%%%%%%%%%%%%%%%%%%%%%%%%%%%%%%%%%%%

\section{Proofs of Theorem \ref{generatorsW} and of Proposition \ref{proposition3vinI'} }

\begin{proof}[Proof of Theorem \ref{generatorsW}]

Let $I$ be the two-sided ideal of $A$ generated by the elements of the forms (\ref{generator[sxx]})--(\ref{generators[xx][xx]}). We need to prove that $W=I$.

\medskip
First we prove that $I \subseteq W$. It suffices to check that all elements of the forms (\ref{generator[sxx]})--(\ref{generators[xx][xx]}) belong to $W$.

We need the following.

%%%%%%%%%%%%%%%%%%%%%%%%%%%%%%%%
\begin{lemma}
\label{lemmau[aaa]inW}
For all $u \in U$, $a_i \in A$, we have  $u [a_1, a_2, a_3] \in W$.
\end{lemma}

%%%%%%%%%%%%%%%%%%%%%%%%%%%%%%%%
\begin{proof}
We have
\[
\bigl[ u a_3, [a_1, a_2] \bigr] = \bigl[ u, [a_1, a_2] \bigr] a_3 + u \bigl[ a_3, [a_1, a_2] \bigr] .
\]
On the other hand, for all $v \in U$, $a_i \in A$,
\[
\bigl[ v, [a_1, a_2] \bigr] = [v, a_1, a_2] - [v, a_2, a_1] \in W.
\]
In particular, $\bigl[ u a_3, [a_1, a_2] \bigr] \in W$ and $\bigl[ u, [a_1, a_2] \bigr] \in W$. It follows that
$u \bigl[ a_3, [a_1, a_2] \bigr]  \in W$ so
\[
u [a_1, a_2, a_3] = - u \bigl[ a_3, [a_1, a_2] \bigr]  \in W ,
\]
as required.
\end{proof}

%%%%%%%%%%%%%%%%%%%%%%%%%%%%%%%%
\begin{corollary}
\label{corollaryu[aa][aa]inW}
For all $u \in U$, $a_i \in A$, we have  $u \bigl( [a_1, a_2] [a_3, a_4]  + [a_1, a_3] [ a_2, a_4] \bigr) \in W$.
\end{corollary}

%%%%%%%%%%%%%%%%%%%%%%%%%%%%%%%%
\begin{proof}
We have
\[
u [a_1 a_4, a_2, a_3] = u \bigl( [a_1, a_2, a_3] a_4 + [a_1, a_2] [a_4, a_3] + [a_1, a_3] [a_4, a_2 ] + a_1 [a_4, a_2, a_3] \bigr) .
\]
By Lemma \ref{lemmau[aaa]inW}, $u [a_1 a_4, a_2, a_3] \in W$, $ u  [a_1, a_2, a_3] \in W$ and $u a_1 [a_4, a_2, a_3] \in W$. Hence,
\[
u \bigl( [a_1, a_2] [a_4, a_3] + [a_1, a_3] [a_4, a_2 ] \bigr) \in W.
\]
It follows that $u \bigl( [a_1, a_2] [a_3, a_4]  + [a_1, a_3] [ a_2, a_4] \bigr) = - u \bigl( [a_1, a_2] [a_4, a_3] + [a_1, a_3] [a_4, a_2 ] \bigr) \in W$, as required.
\end{proof}

Now we are in a position to prove that all elements of the forms (\ref{generator[sxx]})--(\ref{generators[xx][xx]}) belong to $W$. It is clear that each element $[s, x_1, x_2]$ $(s \in S, x_i \in X)$ of the form (\ref{generator[sxx]}) belongs to $W$. By Lemma \ref{lemmau[aaa]inW}, the elements of the forms (\ref{generators[xxx]}) and (\ref{generator[sx][xxx]}) belong to $W$ and, by Corollary  \ref{corollaryu[aa][aa]inW}, so do the elements of the form (\ref{generators[xx][xx]}). Finally, we have
\[
[s x_3 , x_1, x_2] = [s , x_1, x_2] x_3 + [s, x_1] [x_3, x_2] + [s, x_2][x_3, x_1] + s [x_3, x_1, x_2]
\]
so
\[
[s, x_1] [x_2, x_3] + [s, x_2][x_1, x_3] = - [s, x_1] [x_3, x_2] - [s, x_2][x_3, x_1] = - [s x_3 , x_1, x_2]  + [s , x_1, x_2] x_3 +  s [x_3, x_1, x_2] .
\]
Since $[s x_3 , x_1, x_2] \in W$, $[s , x_1, x_2] \in W$ and $s [x_3, x_1, x_2] \in W$, each element $[s, x_1] [x_2, x_3] + [s, x_2][x_1, x_3] $ $(s \in S, x_i \in X)$ of the form (\ref{generator[sx][xx]}) belongs to $W$.

Thus, all elements of the forms (\ref{generator[sxx]})--(\ref{generators[xx][xx]}) belong to $W$ so $I \subseteq W$.

\medskip
Now we check that $W \subseteq I$.

Let $M$ be the set of all products of elements of $X$,
\[
M = \{ x_1 \dots x_k \mid k \ge 0, x_i \in X \} .
\]
Note that the ideal $W$ is generated (as a two-sided ideal in $A$) by all commutators of the form
\[
[b_1 s b_2, b_3, x] \qquad (s \in S, b_i \in M, x \in X ) .
\]
Indeed, since each element of $U$ is a sum of elements of the form $a_1 s a_2$ where $s \in S$, $a_i \in A$, the ideal $W$ is generated by the commutators $[a_1 s a_2, a_3, a_4]$ $(s \in S, a_i \in A)$. Further, each element of $A$ is an $R$-linear combination of element of $M$ so the ideal $W$ is generated by all commutators of the form $[b_1 s b_2, b_3, b_4]$ $(s \in S, b_i \in M)$. Finally,  if $b_4 = y_1 \dots y_k$ where $k \ge 1$, $y_i \in X$ then
\[
[b_1 s b_2, b_3, b_4] = [b_1 s b_2, b_3, y_1 \dots y_k] = \sum_{i=1}^k y_1 \dots y_{i-1} [b_1 s b_2, b_3, y_i] y_{i+1} \dots y_k .
\]
It follows that $W$ is generated by the commutators $[b_1 s b_2, b_3, x] $ $(s \in S, b_i \in M, x \in X)$, as claimed.

Thus, to check that $W \subseteq I$ it suffices to prove that
\begin{equation}
\label{[bsbbx]inI}
[b_1 s b_2, b_3, x] \equiv 0 \pmod{I} \qquad ( s \in S, b_i \in M, x \in X) .
\end{equation}
Further, we have
\[
[b_1 s b_2, b_3, x] = [ s b_1 b_2, b_3, x] - \bigl[ [s, b_1] b_2, b_3, x] \bigr] .
\]
so to prove (\ref{[bsbbx]inI})  is suffices to check that
\begin{equation}
\label{[sddx]}
[ s d_1, d_2, x] \equiv 0 \pmod{I}  \qquad (s \in S, d_i \in M, x \in X)
\end{equation}
and
\begin{equation}
\label{[[sb]bbx]}
\bigl[ [s, b_1] b_2, b_3, x] \bigr] \equiv 0 \pmod{I} \qquad (s \in S, b_i \in M, x \in X).
\end{equation}

First we check that (\ref{[[sb]bbx]}) holds. We need the following observations and lemmas.

Since $[s, x_1, x_2] \in I$, we have $[s, x_1] x_2 \equiv x_2 [s, x_1] \pmod{I}$ for all $s \in S$, $x_i \in X$ so
\begin{equation}
\label{[sx]center}
[s, x] \ a \equiv a \ [s, x] \pmod{I}
\end{equation}
for all $a \in A$, $s \in S$, $x \in X$. Further,
\begin{equation*}
\bigl[ [s, x_1] [x_2, x_3] , x_4 \bigr] = [s, x_1, x_4] [x_2, x_3] + [s, x_1] [x_2, x_3, x_4] \in I
\end{equation*}
for all $s \in S$, $x_i \in X$. It follows that
\begin{equation}
\label{[sx][xx]center}
\bigl( [s, x_1] [x_2, x_3] \bigr) a \equiv a \bigl( [s, x_1] [x_2, x_3] \bigr)  \pmod{I}
\end{equation}
for all $a \in A$, $s \in S$, $x_i \in X$.

%%%%%%%%%%%%%%%%%%%%%%%%%%%%%%%%
\begin{lemma}
\label{lemma[sx][xx][xx]inI}
For all $s \in S$, $x_i \in X$, we have $[s, x_1] \bigl( [x_2, x_3] [x_4, x_5] + [x_2, x_4][x_3, x_5] \bigr) \equiv 0 \pmod{I}$.
\end{lemma}

%%%%%%%%%%%%%%%%%%%%%%%%%%%%%%%%
\begin{proof}
By (\ref{[sx]center}), we have
\begin{align*}
& [s, x_1] \bigl( [x_2, x_3] [x_4, x_5] + [x_2, x_4][x_3, x_5] \bigr) = - [s, x_1] \bigl( [x_3, x_2] [x_4, x_5] + [x_4, x_2][x_3, x_5] \bigr)
\\
\equiv \ & - \bigl( [s, x_1][x_3, x_2] + [s, x_3][x_1, x_2] \bigr) [x_4, x_5] + [x_1, x_2] \bigl( [s, x_3] [ x_4, x_5] + [s, x_4] [x_3, x_5] \bigr)
\\
- \ & \bigl( [s, x_4] [x_1, x_2] + [s, x_1] [x_4, x_2] \bigr) [x_3, x_5] \pmod{I}.
\end{align*}
Note that the element $[s, x_1][x_3, x_2] + [s, x_3][x_1, x_2]$ is of the form (\ref{generator[sx][xx]}) so
\[
[s, x_1][x_3, x_2] + [s, x_3][x_1, x_2] \equiv 0 \pmod{I}.
\]
Similarly,
\[
[s, x_3] [ x_4, x_5] + [s, x_4] [x_3, x_5]  \equiv 0 \pmod{I}
\]
and
\[
[s, x_4] [x_1, x_2] + [s, x_1] [x_4, x_2] \equiv 0 \pmod{I}.
\]
The result follows.
\end{proof}

%%%%%%%%%%%%%%%%%%%%%%%%%%%%%%%%%%
\begin{lemma}
\label{lemma[sx][bxx]}
For all $s \in S$, $b \in M$, $x_i \in X$, we have $[s, x_1] [b, x_2, x_3] \equiv 0 \pmod{I}$.
\end{lemma}

%%%%%%%%%%%%%%%%%%%%%%%%%%%%%%%%%%
\begin{proof}
Suppose that $b = y_1 \dots y_k$ $(k \ge 1, y_i \in X)$. We have
\begin{align*}
& [s, x_1] [b, x_2, x_3] = [s, x_1] [y_1 \dots y_k, x_2, x_3] = \sum_{i=1}^k [s, x_1] y_1 \dots y_{i-1} [y_i, x_2, x_3] y_{i+1} \dots y_k
\\
+ \ & \sum_{1 \le i < i' \le k} [s, x_1] y_1 \dots y_{i-1} \bigl( [y_i, x_2] y_{i+1} \dots y_{i'-1} [y_{i'} , x_3] +  [y_i, x_3] y_{i+1} \dots y_{i'-1} [y_{i'} , x_2] \bigr) y_{i'+1} \dots y_k .
\end{align*}
Note that, by (\ref{[sx]center}),
\[
[s, x_1] y_1 \dots y_{i-1} [y_i, x_2, x_3] y_{i+1} \dots y_k \equiv  y_1 \dots y_{i-1}[s, x_1]  [y_i, x_2, x_3] y_{i+1} \dots y_k \pmod{I}
\]
and, by (\ref{[sx]center}) and (\ref{[sx][xx]center}),
\begin{align*}
& [s, x_1] y_1 \dots y_{i-1} \bigl( [y_i, x_2] y_{i+1} \dots y_{i'-1} [y_{i'} , x_3] +  [y_i, x_3] y_{i+1} \dots y_{i'-1} [y_{i'} , x_2] \bigr) y_{i'+1} \dots y_k
\\
\equiv \ &  y_1 \dots y_{i-1} y_{i+1} \dots y_{i'-1} [s, x_1] \bigl( [y_i, x_2]  [y_{i'} , x_3] +  [y_i, x_3]  [y_{i'} , x_2] \bigr) y_{i'+1} \dots y_k \pmod{I} .
\end{align*}
Since the product $[s, x_1]  [y_i, x_2, x_3]$ is of the form (\ref{generator[sx][xxx]}), we have $[s, x_1]  [y_i, x_2, x_3] \equiv 0 \pmod{I}$; by Lemma \ref{lemma[sx][xx][xx]inI},
\[
[s, x_1] \bigl( [y_i, x_2]  [y_{i'} , x_3] +  [y_i, x_3]  [y_{i'} , x_2] \bigr) \equiv 0 \pmod{I} .
\]
Therefore, $[s, x_1] [b, x_2, x_3] \equiv 0 \pmod{I}$ for all $s \in S$, $b \in M$, $x_i \in X$, as required.
\end{proof}

%%%%%%%%%%%%%%%%%%%%%%%%%%%%%%%%%%
\begin{corollary}
\label{corollary[sx][bx]a}
For all $s \in S$, $a \in A$, $b \in M$, $x_j \in X$, we have
\[
[s, x_1] [b, x_2, ] a \equiv a [s, x_1] [b, x_2] \pmod{I} .
\]
\end{corollary}

%%%%%%%%%%%%%%%%%%%%%%%%%%%%%%%%%%
\begin{proof} It suffices to check that, for all $s \in S$, $b \in M$, $x_i \in X$,
\[
[s, x_1] [b, x_2, ] x_3 \equiv x_3 [s, x_1] [b, x_2] \pmod{I} ,
\]
that is,
\[
\bigl[ [s, x_1] [b, x_2, ] , x_3 \bigr] \equiv 0 \pmod{I}.
\]
We have
\[
\bigl[ [s, x_1] [b, x_2, ] , x_3 \bigr] = [s, x_1, x_3] [b, x_2] + [s, x_1] [b, x_2, x_3].
\]
Here $[s, x_1, x_3] \equiv 0 \pmod{I}$ because it is an element of the form (\ref{generator[sxx]}) and $[s, x_1] [b, x_2, x_3] \equiv 0 \pmod{I}$ by Lemma \ref{lemma[sx][bxx]}. The result follows.
\end{proof}

Now we are in a position to check that (\ref{[[sb]bbx]}) holds.
%%%%%%%%%%%%%%%%%%%%%%%%%%%%%%%%%%
\begin{lemma}
For all $s \in S$, $b_i \in M$, $x \in X$, we have $\bigl[ [s, b_1] b_2, b_3, x \bigr] \equiv 0 \pmod{I}$.
\end{lemma}

%%%%%%%%%%%%%%%%%%%%%%%%%%%%%%%%%%
\begin{proof}
Let $b_1 = y_1 \dots y_k$ $(k \ge 1, y_j \in X)$. Then
\[
\bigl[ [s, b_1] b_2, b_3, x \bigr] = \bigl[ [s, y_1 \dots y_k] b_2, b_3, x \bigr] = \sum_{i=1}^k [ y_1 \dots y_{i-1} [s, y_i] y_{i+1} \dots y_k b_2, b_3, x ] .
\]
By (\ref{[sx]center}),
\[
[ y_1 \dots y_{i-1} [s, y_i] y_{i+1} \dots y_k b_2, b_3, x ] \equiv  [s, y_i]  [ y_1 \dots y_{i-1}y_{i+1} \dots y_k b_2, b_3, x ] \pmod{I}
\]
so to prove the lemma it suffices to check that, for all $s \in S$, $d_i \in M$, $x_i \in X$,
\begin{equation}
\label{[sx][ddx]}
[s, x_1] [d_1, d_2, x_2] \equiv 0 \pmod{I} .
\end{equation}
By Lemma \ref{L2}, each commutator $[d_1, d_2, x]$ $(d_i \in M, x \in X)$ is a linear combination of commutators of the form $[d, x_1, x_2]$ where $d \in M$, $x_i \in X$. It follows that to prove (\ref{[sx][ddx]}) it suffices to check that, for all $s \in S$, $d \in M$, $x_i \in X$, we have $[s, x_1] [d, x_2, x_3] \equiv 0 \pmod{I}$ and this congruence holds by Lemma \ref{lemma[sx][bxx]}.
\end{proof}

Thus, (\ref{[[sb]bbx]}) holds, as required.

\medskip
Now we check that (\ref{[sddx]}) holds. We have
\begin{equation}
\label{[sddx]=sum}
[ s d_1, d_2, x]  = [s, d_2, x] d_1 + [s, d_2] [d_1, x] + [s, x] [d_1, d_2] + s [d_1, d_2, x].
\end{equation}
We shall prove that, for each $s \in S$, $d, d_i \in M$, $x \in X$,
\[
[s, d, x] \equiv 0 \pmod{I}, \quad [s, d_2] [d_1, x] + [s, x] [d_1, d_2] \equiv 0 \pmod{I} \quad \mbox{and}  \quad s [d_1, d_2, x] \equiv 0 \pmod{I}.
\]

%%%%%%%%%%%%%%%%%%%%%%%%%%%%%%%%%%%%
\begin{lemma}
\label{lemma[sdx]}
For all $s \in S$, $d \in M$, $x \in X$, we have $[s, d, x] \equiv 0 \pmod{I}$.
\end{lemma}

%%%%%%%%%%%%%%%%%%%%%%%%%%%%%%%%%%%%
\begin{proof}
Let $d = y_1 \dots y_k$ where $k \ge 1$, $y_i \in X$. We have
\begin{align*}
&  [d, s, x] =  [y_1 \dots y_k, s, x] =  \sum_{i=1}^k y_1 \dots y_{i-1} [y_i, s, x] y_{i+1} \dots y_k
\\
+ \ & \sum_{1 \le i < i' \le k} y_1 \dots y_{i-1} \bigl( [y_i, s] y_{i+1} \dots y_{i' - 1} [y_{i'} , x] + [y_i, x] y_{i+1} \dots y_{i' - 1} [y_{i'} , s] \bigr) y_{i' +1} \dots y_k
\end{align*}
Note that $[s, y_i, x]$ is an element of the form (\ref{generator[sxx]}) so, for all $i$,
\[
[y_i, s, x] = - [s, y_i, x] \equiv 0 \pmod{I}.
\]
On the other hand, by (\ref{[sx]center}) and (\ref{[sx][xx]center}),
\begin{align*}
& [y_i, s] y_{i+1} \dots y_{i' - 1} [y_{i'} , x] + [y_i, x] y_{i+1} \dots y_{i' - 1} [y_{i'} , s]
\\
\equiv \ & - \bigl( [s, y_i] [y_{i'} , x]   + [s, y_{i'}] [y_i, x] \bigr)  y_{i+1} \dots y_{i' - 1} \pmod{I} .
\end{align*}
Since, for all $i, i'$,  $ [s, y_i] [y_{i'} , x]   + [s, y_{i'}] [y_i, x]$ is an element of the form (\ref{generator[sx][xx]}), we have
\[
[s, y_i] [y_{i'} , x]   + [s, y_{i'}] [y_i, x] \equiv 0 \pmod{I}.
\]
It follows that $[d,s,x] \equiv 0 \pmod{I}$ so $[s, d, x] = - [d,s,x] \equiv 0 \pmod{I}$, as required.
\end{proof}

%%%%%%%%%%%%%%%%%%%%%%%%%%%%%%%%%%%%
\begin{lemma}
\label{lemma[sx][dd]}
For all $s \in S$, $d_i \in M$, $x \in X$, we have $[s, x] [d_1, d_2] + [s, d_2][d_1, x] \equiv 0 \pmod{I}$.
\end{lemma}

%%%%%%%%%%%%%%%%%%%%%%%%%%%%%%%%%%%%
\begin{proof}
Let $d_2 = y_1 \dots y_k$ where $y_i \in X$. We have
\begin{align*}
& [s, x] [d_1, d_2] + [s, d_2][d_1, x] = [s, x] [d_1, y_1 \dots y_k] + [s, y_1 \dots y_k][d_1, x]
\\
& = \  \sum_{i=1}^k \bigl( [s, x] y_1 \dots y_{i-1} [d_1, y_i ] y_{i+1} \dots y_k + y_1 \dots y_{i-1} [s, y_i] y_{i+1} \dots y_k [d_1, x] \bigr) .
\end{align*}
By (\ref{[sx]center}) and Corollary \ref{corollary[sx][bx]a},
\begin{align*}
 & [s, x] y_1 \dots y_{i-1} [d_1, y_i ] y_{i+1} \dots y_k + y_1 \dots y_{i-1} [s, y_i] y_{i+1} \dots y_k [d_1, x]
 \\
 & \equiv  y_1 \dots y_{i-1}  y_{i+1} \dots y_k \bigl( [s, x] [d_1, y_i ] +  [s, y_i]  [d_1, x] \bigr) \pmod{I} .
\end{align*}
Suppose that $d_1 = z_1 \dots z_{\ell}$. We have
\begin{align*}
& [s, x] [d_1, y_i ] +  [s, y_i]  [d_1, x] = [s, x] [z_1 \dots z_{\ell}, y_i ] +  [s, y_i]  [z_1 \dots z_{\ell}, x]
\\
= \ & \sum_{j=1}^{\ell}  \bigl(  [s, x] z_1 \dots z_{j - 1} [z_j , y_i ] z_{j+1} \dots z_{\ell} +  [s, y_i] z_1 \dots z_{j - 1} [z_j , x] z_{j + 1} \dots z_{\ell} \bigr)
\end{align*}
where, by (\ref{[sx]center}) and (\ref{[sx][xx]center}),
\begin{align*}
& [s, x] z_1 \dots z_{j - 1} [z_j , y_i ] z_{j+1} \dots z_{\ell} +  [s, y_i] z_1 \dots z_{j - 1} [z_j , x] z_{j + 1} \dots z_{\ell}
\\
& \equiv \ z_1 \dots z_{j - 1} z_{j+1} \dots z_{\ell} \bigl(  [s, x]  [z_j , y_i ]  +  [s, y_i]  [z_j , x]  \bigr) \pmod{I} .
\end{align*}
Since $[s, x]  [y_i, z_j ]  +  [s, y_i]  [x, z_j]$ is an element of the form (\ref{generator[sx][xx]}), we have
\[
[s, x]  [z_j , y_i ]  +  [s, y_i]  [z_j , x] = - \bigl( [s, x]  [y_i, z_j ]  +  [s, y_i]  [x, z_j] \bigr) \equiv 0 \pmod{I} .
\]
The result follows.
\end{proof}

%%%%%%%%%%%%%%%%%%%%%%%%%%%%%%%%%%%%
\begin{lemma}
\label{lemmas[ddx]}
For all $s \in S$, $d_i \in M$, $x \in X$, we have $s [d_1, d_2, x] \equiv 0 \pmod{I}$.
\end{lemma}

%%%%%%%%%%%%%%%%%%%%%%%%%%%%%%%%%%%%
\begin{proof}
By Lemma \ref{L2}, each commutator $[d_1, d_2, x]$ $(d_i \in M, x \in X)$ is a linear combination of commutators of the form $[d, x_1, x_2]$ where $d \in M$, $x_i \in X$. Hence, to prove the lemma it suffices to check that, for all $s \in S$, $d \in M$, $x_i \in X$,
\[
s [d, x_1, x_2] \equiv 0 \pmod{I}.
\]
Let $d = y_1 \dots y_k$ $( k \ge 1, y_j \in X)$. Then
\begin{equation}
\label{s[dxx]Q1Q2}
 s [d, x_1, x_2] = s [ y_1 \dots y_k, x_1, x_2] = Q_1 + Q_2
\end{equation}
where
\begin{align}
& Q_1 = \  \sum_{i = 1}^k s y_1 \dots y_{i-1} [y_i, x_1, x_2] y_{i+1} \dots y_k , \nonumber
\\
\label{Q2}
& Q_2 =  \  \sum_{1 \le i <i' \le k} s y_1 \dots y_{i-1} \bigl( [y_i, x_1] y_{i+1} \dots y_{i' - 1} [y_{i'} , x_2] + [y_i, x_2] y_{i+1} \dots y_{i' - 1} [y_{i'} , x_1] \bigr) y_{i'+1} \dots y_k .
\end{align}
First we check that $Q_1 \equiv 0 \pmod{I}$. We have
\[
s y_1 \dots y_{i-1} = y_1 \dots y_{i-1} s + [s, y_1 \dots y_{i-1}]  = y_1 \dots y_{i-1} s + \sum_{j=1}^{i-1} y_1 \dots y_{j-1} [s, y_j] y_{j+1} \dots y_{i-1}
\]
so
\[
s y_1 \dots y_{i-1} [y_i, x_1, x_2] = y_1 \dots y_{i-1} s [y_i, x_1, x_2] + \sum_{j=1}^{i-1} y_1 \dots y_{j-1} [s, y_j] y_{j+1} \dots y_{i-1} [y_i, x_1, x_2] .
\]
Note that $s [y_i, x_1, x_2]$ is an element of the form  (\ref{generators[xxx]}) so $s [y_i, x_1, x_2] \equiv 0 \pmod{I}$. On the other hand, by (\ref{[sx]center}),
\[
[s, y_j] y_{j+1} \dots y_{i-1} [y_i, x_1, x_2]  \equiv  y_{j+1} \dots y_{i-1} [s, y_j] [y_i, x_1, x_2]  \pmod{I}
\]
where $[s, y_j] [y_i, x_1, x_2]  \equiv 0 \pmod{I}$ because $[s, y_j] [y_i, x_1, x_2] $ is an element of the form (\ref{generator[sx][xxx]}). It follows that, for each $i$,
\[
s y_1 \dots y_{i-1} [y_i, x_1, x_2] \equiv 0 \pmod{I}
\]
and, therefore,
\begin{equation}
\label{Q10}
Q_1 \equiv 0 \pmod{I}.
\end{equation}

Now we check that $Q_2 \equiv 0 \pmod{I}$. We have
\begin{equation}
\label{Q2Q1'Q2'}
s y_1 \dots y_{i-1} \bigl( [y_i, x_1] y_{i+1} \dots y_{i' - 1} [y_{i'} , x_2] + [y_i, x_2] y_{i+1} \dots y_{i' - 1} [y_{i'} , x_1] \bigr) = Q_1' + Q_2'
\end{equation}
where
\begin{align}
\label{Q1'}
Q_1' = \ &  y_1 \dots y_{i-1} s \bigl( [y_i, x_1] y_{i+1} \dots y_{i' - 1} [y_{i'} , x_2] + [y_i, x_2] y_{i+1} \dots y_{i' - 1} [y_{i'} , x_1] \bigr) ,
\\
Q_2' = \ & [s,  y_1 \dots y_{i-1}] \bigl( [y_i, x_1] y_{i+1} \dots y_{i' - 1} [y_{i'} , x_2] + [y_i, x_2] y_{i+1} \dots y_{i' - 1} [y_{i'} , x_1] \bigr) \nonumber
\\
=  \ & \sum_{j=1}^{i-1} y_1 \dots y_{j-1} [s, y_j] y_{j+1} \dots y_{i-1} \bigl( [y_i, x_1] y_{i+1} \dots y_{i' - 1} [y_{i'} , x_2] + [y_i, x_2] y_{i+1} \dots y_{i' - 1} [y_{i'} , x_1] \bigr) . \nonumber
\end{align}
By (\ref{[sx]center}) and (\ref{[sx][xx]center}), we have
\begin{align*}
& [s, y_j] y_{j+1} \dots y_{i-1} \bigl( [y_i, x_1] y_{i+1} \dots y_{i' - 1} [y_{i'} , x_2] + [y_i, x_2] y_{i+1} \dots y_{i' - 1} [y_{i'} , x_1] \bigr)
\\
\equiv \ &  y_{j+1} \dots y_{i-1} y_{i+1} \dots y_{i' - 1} [s, y_j] \bigl( [y_i, x_1]  [y_{i'} , x_2] + [y_i, x_2]  [y_{i'} , x_1] \bigr) \pmod{I}
\end{align*}
where, by Lemma \ref{lemma[sx][xx][xx]inI},
\[
[s, y_j] \bigl( [y_i, x_1]  [y_{i'} , x_2] + [y_i, x_2]  [y_{i'} , x_1] \bigr) \equiv 0 \pmod{I} .
\]
It follows that
\begin{equation}
\label{Q2'0}
Q_2' \equiv 0 \pmod{I}.
\end{equation}

Now we check that $Q_1' \equiv 0 \pmod{I}$. We have
\begin{align}
& s \bigl( [y_i, x_1] y_{i+1} \dots y_{i' - 1} [y_{i'} , x_2] + [y_i, x_2] y_{i+1} \dots y_{i' - 1} [y_{i'} , x_1] \bigr)  \nonumber
\\
\label{Q1'sum}
= \ & y_{i+1} \dots y_{i' - 1} s \bigl( [y_i, x_1]  [y_{i'} , x_2] + [y_i, x_2]  [y_{i'} , x_1] \bigr) + \bigl[ s  [y_i, x_1], y_{i+1} \dots y_{i' - 1} \bigr]  [y_{i'} , x_2]
\\
+ \ &  \bigl[ s  [y_i, x_2], y_{i+1} \dots y_{i' - 1} \bigr]  [y_{i'} , x_1] . \nonumber
\end{align}
Since $s \bigl( [y_i, x_1]  [y_{i'} , x_2] + [y_i, x_2]  [y_{i'} , x_1] \bigr)$ is an element of the form (\ref{generators[xx][xx]}), we have
\begin{equation}
\label{Q1'firstterm}
s \bigl( [y_i, x_1]  [y_{i'} , x_2] + [y_i, x_2]  [y_{i'} , x_1] \bigr) \equiv 0 \pmod{I} .
\end{equation}
On the other hand,
\begin{align*}
& \bigl[ s  [y_i, x_j], y_{i+1} \dots y_{i' - 1} \bigr] = \sum_{m=i+1}^{i'-1} y_{i+1} \dots y_{m-1} \bigl[ s [y_i, x_j], y_m \bigr] y_{m+1} \dots y_{i'-1}
\\
= \ & \sum_{m=i+1}^{i'-1} y_{i+1} \dots y_{m-1} \bigl( [s, y_m] [y_i, x_j] + s[y_i, x_j, y_m] \bigr) y_{m+1} \dots y_{i'-1}
\end{align*}
Since $s[y_i, x_j, y_m]$ is an element of the form (\ref{generators[xxx]}), we have $s[y_i, x_j, y_m] \equiv 0 \pmod{I}$. Hence,
\[
 \bigl[ s  [y_i, x_j], y_{i+1} \dots y_{i' - 1} \bigr] \equiv  \sum_{m=i+1}^{i'-1} y_{i+1} \dots y_{m-1} [s, y_m] [y_i, x_j]  y_{m+1} \dots y_{i'-1} \pmod{I}
\]
and
\begin{align*}
& \bigl[ s  [y_i, x_1], y_{i+1} \dots y_{i' - 1} \bigr]  [y_{i'} , x_2]  +   \bigl[ s  [y_i, x_2], y_{i+1} \dots y_{i' - 1} \bigr]  [y_{i'} , x_1]
\\
& \equiv \  \sum_{m=i+1}^{i'-1} y_{i+1} \dots y_{m-1} \bigl( [s, y_m] [y_i, x_1]  y_{m+1} \dots y_{i'-1} [y_{i'} , x_2]
\\
& + \  [s, y_m] [y_i, x_2]  y_{m+1} \dots y_{i'-1} [y_{i'} , x_1] \bigr) \pmod{I}.
\end{align*}
By  (\ref{[sx][xx]center}), we have
\begin{align*}
& [s, y_m] [y_i, x_1]  y_{m+1} \dots y_{i'-1} [y_{i'} , x_2]  + [s, y_m] [y_i, x_2]  y_{m+1} \dots y_{i'-1} [y_{i'} , x_1]
\\
& \equiv \ y_{m+1} \dots y_{i'-1}   [s, y_m] \bigl( [y_i, x_1]   [y_{i'} , x_2]  +  [y_i, x_2]   [y_{i'} , x_1] \bigr) \pmod{I} .
\end{align*}
where, by Lemma \ref{lemma[sx][xx][xx]inI},
\[
[s, y_m] \bigl( [y_i, x_1]   [y_{i'} , x_2]  +  [y_i, x_2]   [y_{i'} , x_1] \bigr) = - [s, y_m] \bigl( [y_i, x_1]   [x_2, y_{i'} ]  +  [y_i, x_2]   [x_1, y_{i'} ] \bigr) \equiv 0 \pmod{I}.
\]
Hence,
\[
 \bigl[ s  [y_i, x_1], y_{i+1} \dots y_{i' - 1} \bigr]  [y_{i'} , x_2]  +   \bigl[ s  [y_i, x_2], y_{i+1} \dots y_{i' - 1} \bigr]  [y_{i'} , x_1] \equiv 0 \pmod{I}
\]
so, by (\ref{Q1'}), (\ref{Q1'sum}), (\ref{Q1'firstterm}) and the congruence above, $Q_1' \equiv 0 \pmod{I}$. It follows that,  by (\ref{Q2}), (\ref{Q2Q1'Q2'}) and (\ref{Q2'0}),  $Q_2 \equiv 0 \pmod{I}$, therefore, by (\ref{s[dxx]Q1Q2}) and (\ref{Q10}), $s [d, x_1, x_2] \equiv 0 \pmod{I}$ for all $s \in S$, $d \in M$, $x_i \in X$. The result follows.
\end{proof}

Thus, by (\ref{[sddx]=sum}) and Lemmas \ref{lemma[sdx]}, \ref{lemma[sx][dd]} and \ref{lemmas[ddx]}, (\ref{[sddx]}) holds. It follows that  (\ref{[bsbbx]inI}) holds and, therefore,  $W \subseteq I$. This  completes the proof of Theorem \ref{generatorsW}.
\end{proof}

%%%%%%%%%%%%%%%%%%%%%%%%%%%%%%%%%%%%
\begin{proof}[Proof of Proposition \ref{proposition3vinI'} ]
Recall that $I'$ is the two-sided ideal of $A$ generated by all elements of the forms (\ref{generator[sxx]}) and   (\ref{generator[sx][xx]}). It is clear that, for all $s \in S$, $x_m \in X$,
\begin{equation*}
\label{[sx]x}
[s, x_i] x_j \equiv x_j [s, x_i] \pmod{I'} \qquad
\mbox{ and }
\qquad  [s, x_i][x_j, x_k] \equiv - [s, x_j] [x_i, x_k] \pmod{I'} .
\end{equation*}
It follows that
\begin{align*}
& [s, x_1][ x_2, x_3, x_4] = [s, x_1][ x_2, x_3] x_4 -  [s, x_1]x_4 [ x_2, x_3]
\\
\equiv \ & [s, x_1][ x_2, x_3] x_4 -  x_4 [s, x_1] [ x_2, x_3] \pmod{I'}  \equiv - [s, x_2][ x_1, x_3] x_4 +  x_4 [s, x_2] [ x_1, x_3] \pmod{I'}
\\
\equiv \ & - [s, x_2][ x_1, x_3] x_4 +  [s, x_2] x_4 [ x_1, x_3] \pmod{I'} = - [s, x_2][ x_1, x_3, x_4] ,
\end{align*}
that is, 
\[
[s, x_1][ x_2, x_3, x_4]  + [s, x_2][ x_1, x_3, x_4] \equiv 0 \pmod{I'}
\]
for all $s \in S$, $x_i \in X$. By Corollary \ref{corollary3[sz][zzz]}, we have $ 3 \, [s, x_1] [x_2, x_3, x_4] \equiv 0 \pmod{I'}$ for all $s \in S$, $x_i \in X$, as required.

The proof of Proposition \ref{proposition3vinI'} is completed.
\end{proof}

%%%%%%%%%%%%%%%%%%%%%%%%%%%%%%%%%%%%
%%%%%%%%%%%%%%%%%%%%%%%%%%%%%%%%%%%%
%%%%%%%%%%%%%%%%%%%%%%%%%%%%%%%%%%%%
%%%%%%%%%%%%%%%%%%%%%%%%%%%%%%%%%%%%
%%%%%%%%%%%%%%%%%%%%%%%%%%%%%%%%%%%%

\section{Proof of Theorem \ref{generatorsTn} }

Let $I^{(m)}$ $(m \ge 2)$ be the two-sided ideal of $A$ generated by the set
\[
S^{(m)} = \{ [y_1, y_2, \dots , y_m] \mid y_1, y_m \in X; \ y_2, \dots , y_{m-1} \in X \cup X^2 \} .
\]
Define $S^{(1)} = \{ 1 \}$, then $I^{(1)} = A$.

We have to prove that, for each $n \ge 1$, $T^{(n)} = I^{(n)}$.  We will prove this by induction on $n$. If $n = 1$ then $T^{(1)} = A = I^{(1)}$. If $n = 2$  then $S^{(2)} = \{ [x_1, x_2] \mid x_i \in X \}$ so $I^{(2)} = T^{(2)}$.

Suppose that $n>2$ and  $T^{(n-2)} = I^{(n-2)}$; we will check that $T^{(n)} = I^{(n)}$. It is clear that $I^{(n)} \subseteq T^{(n)}$ so it remains to prove that $T^{(n)} \subseteq I^{(n)}$. 

By Proposition \ref{A[Tn-2AA]A=Tn}, $T^{(n)}$ is generated as a two-sided ideal in $A$ by the set 
\[
\{ [u, a_1, a_2] \mid u \in T^{(n-2)}, a_1, a_2 \in A \} .
\]
Hence, by Corollary \ref{corollarygeneratorsW},  the ideal $T^{(n)}$ is generated by the following elements:
\begin{equation}
\label{generatorTn[sxx]}
[s, x_1, x_2] \qquad (s \in S^{(n-2)}, x_i \in X),
\end{equation}
\begin{equation}
\label{generatorTns[xxx]}
s [x_1, x_2, x_3] \qquad (s \in S^{(n-2)}, x_i \in X ),
\end{equation}
\begin{equation}
\label{generatorTn[sx][xx]}
[s, x_1][x_2,x_3] + [s, x_2][x_1, x_3] \qquad (s \in S^{(n-2)}, x_i \in X ),
\end{equation}
\begin{equation}
\label{generatorTns[xx]xx]}
s \bigl( [x_1, x_2][x_3, x_4] + [x_1, x_3][x_2, x_4] \bigr) \qquad (s \in S^{(n-2)}, x_i \in X ) .
\end{equation}
It follows that to prove that $T^{(n)} \subseteq I^{(n)}$ it suffices to check that all elements (\ref{generatorTn[sxx]})--(\ref{generatorTns[xx]xx]}) belong to $I^{(n)}$.
%%%%%%%%%%%%%%%%%%%%%%%%%%%%%%%%%%%%
%%%%%%%%%%%%%%%%%%%%%%%%%%%%%%%%%%%%

\medskip
First, we note that the elements of the form (\ref{generatorTn[sxx]}) belong to $I^{(n)}$ because they belong to $S^{(n)}$.

%%%%%%%%%%%%%%%%%%%%%%%%%%%%%%%%%%%%
Further, suppose that $d = [y_1, \dots , y_{n-2}]$ where $y_1 \in X$, $y_2, \dots , y_{n-2} \in X \cup X^2$. Then 
\[
[d, x_1, x_2], \ [d, x_1 x_2, x_3] \in S^{(n)}
\]
so $[d, x_1, x_2] \equiv 0 \pmod{I^{(n)} }$ and $[d, x_1 x_2, x_3] \equiv 0 \pmod{I^{(n)} }$ for all $x_i \in X$. By Lemma \ref{lemma[dz][zz]=0}, with $B = A/I^{(n)}$ and $Z = \{ x + I^{(n)} \mid x \in X \}$, we have 
\begin{equation}
\label{[cn-2x][xx]inIn}
[d, x_1][x_2, x_3] +  [d, x_2 ] [x_1, x_3] \equiv 0 \pmod{I^{(n)} } .
\end{equation}

Note that each $s \in S^{(n-2)}$ is of the form $s = [y_1, y_2, \dots , y_{n-2}]$ where $y_1, y_{n-2} \in X$, $y_2, \dots , y_{n-3} \in X \cup X^2$. Hence, it follows from (\ref{[cn-2x][xx]inIn}) that, in particular, the elements of the form (\ref{generatorTn[sx][xx]}) belong to $I^{(n)}$.

%%%%%%%%%%%%%%%%%%%%%%%%%%%%%%%%%%%
\medskip
Now we prove that the elements of the form (\ref{generatorTns[xxx]}) belong to $I^{(n)}$. Let $c = [y_1, y_2, \dots , y_{n-3}]$ where $y_1 \in X$, $y_2, \dots , y_{n-3}\in X \cup X^2$.  Let $x_i \in X$. Then we have 
\[
[c, x_1, x_2, x_3]  , \  [c, x_1, x_2 x_3, x_4] , \ [c, x_1 x_2, x_3, x_4] \in S^{(n)}
\]
so $[c, x_1, x_2, x_3] \equiv 0 \pmod{I^{(n)} }$, $[c, x_1, x_2 x_3, x_4] \equiv 0 \pmod{I^{(n)} }$, $[c, x_1 x_2, x_3, x_4] \equiv 0 \pmod{I^{(n)} }$. Then, by Lemma \ref{lemma[cz][zzz]=0}, with $B = A/I^{(n)}$ and $Z = \{ x + I^{(n)} \mid x \in X \}$, we have
\begin{equation*}
\label{c12-21}
[c, x_1][x_2, x_3, x_4] + [c, x_2]  [x_1, x_3, x_4]  \equiv 0 \pmod{I^{(n)} }
\end{equation*}
for all $x_i \in X$. Now, by Corollary \ref{corollary3[sz][zzz]}, 
\[
3 \  [c, x_1] [x_2, x_3, x_4] \equiv  0 \pmod{I^{(n)} } .
\]
Since $3$ is invertible in $R$, we have $[c, x_1] [x_2, x_3, x_4] \in I^{(n)} $ for all $x_i \in X$ and all $c = [y_1, y_2, \dots , y_{n-3}]$ where $y_1 \in X$, $y_2, \dots , y_{n-3}\in X \cup X^2$. In other words, we have $s [x_2,x_3,x_4] \in I^{(n)}$ for all $s \in S^{(n-2)}$, $x_i \in X$, that is, all elements of the form (\ref{generatorTns[xxx]}) belong to $I^{(n)}$.

%%%%%%%%%%%%%%%%%%%%%%%%%%%%%%%%
\medskip
Finally, we check that the elements of the form (\ref{generatorTns[xx]xx]}) belong to $I^{(n)}$. Let $g(z_1, z_2, z_3, z_4) = [z_1, z_2] [z_3, z_4] + [z_1, z_3] [z_2, z_4]$. Then each element of the form 
(\ref{generatorTns[xx]xx]}) can be written as 
\[
[c, x_1] g( x_2, x_3, x_4, x_5) 
\]
where $x_i \in X$, $c = [y_1, y_2, \dots , y_{n-3}]$, $y_1 \in X$, $y_2, \dots , y_{n-3}\in X \cup X^2$.

%%%%%%%%%%%%%%%%%%%%%%%%%%%%%%%%%
\begin{lemma}
\label{lemma[c1]g(2345)}
For all $x_i \in X$, 
\[
[c, x_1] g( x_2, x_3, x_4, x_5) \equiv - [c, x_2] g( x_1, x_3, x_4, x_5) \pmod{I^{(n)} } .
\]
\end{lemma}

%%%%%%%%%%%%%%%%%%%%%%%%%%%%%%%%%
\begin{proof}
Note that, by (\ref{[cn-2x][xx]inIn}), 
\[
[c, x_1x_2, x_3][x_4, x_5] +  [c, x_1 x_2, x_4][x_3, x_5] \in I^{(n)} .
\]
We have
\begin{align*}
& [c, x_1x_2, x_i] = - [x_1 x_2, c, x_i] 
\\
= \ & -  x_1 [x_2, c, x_i] - [x_1, c] [x_2, x_i] - [x_1, x_i] [x_2, c]  - [x_1, c, x_i] x_2 
\\
=  \ & x_1 [c, x_2, x_i] + [c, x_1] [x_2, x_i] +  [c, x_2] [x_1, x_i] +  [c, x_1, x_i] x_2 - \bigl[ [c, x_2], [x_1, x_i] \bigr] 
\end{align*}
It follows that 
\begin{align*}
& [c, x_1x_2, x_3][x_4, x_5] +  [c, x_1 x_2, x_4][x_3, x_5]
\\
= \ & \bigl( [c, x_1, x_3] x_2 + [c, x_1][x_2, x_3] +  [c, x_2 ] [x_1, x_3] + x_1 [c, x_2, x_3]
\\
- \ & \bigl[ [c, x_2 ], [x_1, x_3] \bigr] \bigr) [x_4, x_5] + \bigl( [c, x_1, x_4] x_2 + [c, x_1][x_2, x_4] +  [c, x_2 ] [x_1, x_4] 
\\
+ \ & x_1 [c, x_2, x_4] - \bigl[ [c, x_2 ], [x_1, x_4] \bigr] \bigr) [x_3, x_5] \in I^{(n)} .
\end{align*}
Note that $\bigl[ [c, x_i ], [x_j, x_k] \bigr] = [c, x_i, x_j, x_k] - [c, x_i, x_k, x_j] \in I^{(n)}$ and, by (\ref{[cn-2x][xx]inIn}),
\[
\bigl( [c, x_i, x_3] [x_4, x_5] + [c, x_i, x_4] [x_3, x_5] \bigr) \in I^{(n)}
\]
so 
\[
 x_1 \bigl( [c, x_2, x_3] [x_4, x_5] + [c, x_2, x_4] [x_3, x_5] \bigr) \in I^{(n)}
\]
and
\begin{align*}
& [c, x_1, x_3] x_2 [x_4, x_5] + [c, x_1, x_4] x_2 [x_3, x_5] =  x_2 \bigl( [c, x_1, x_3]  [x_4, x_5]
\\
+ \ &  [c, x_1, x_4]  [x_3, x_5] \bigr)  + [c, x_1, x_3, x_2] [x_4, x_5] + [c, x_1, x_4, x_2] [x_3, x_5]  \in I^{(n)} .
\end{align*}
It follows that
\begin{align*}
& \bigl( [c, x_1][x_2, x_3] +  [c, x_2 ] [x_1, x_3] \bigr) [x_4, x_5] + \bigl( [c, x_1][x_2, x_4] +  [c, x_2 ] [x_1, x_4] \bigr) [x_3, x_5]
\\
= \ & [c, x_1] \bigl( [x_2, x_3] [x_4, x_5] + [x_2, x_4] [x_3, x_5] \bigr) + [c, x_2 ] \bigl( [x_1, x_3] [x_4, x_5] + [x_1, x_4] [x_3, x_5] \bigr) \in I^{(n)} ,
\end{align*}
that is, 
\[
[c, x_1] g (x_2, x_3, x_4, x_5) + [c, x_2 ] g( x_1, x_3, x_4, x_5)  \in I^{(n)} .
\]
Thus,
\begin{equation}
[c, x_1] g (x_2, x_3, x_4, x_5) \equiv -  [c, x_2 ] g( x_1, x_3, x_4, x_5)  \pmod{I^{(n)} }  ,
\end{equation}
as required.
\end{proof}

%%%%%%%%%%%%%%%%%%%%%%%%%%%%%%
\begin{lemma}
For all $x_i \in X$, we have $[c, x_1] [x_2, x_3, x_4, x_5] \in I^{(n)} $.
\end{lemma}

%%%%%%%%%%%%%%%%%%%%%%%%%%%%%%
\begin{proof}
We have
\begin{align*}
& [c, x_1] [x_2, x_3, x_4, x_5] = [c, x_1] [x_2, x_3, x_4] x_5 - [c, x_1] x_5 [x_2, x_3, x_4] 
\\
= \ &  [c, x_1] [x_2, x_3, x_4] x_5 - x_5 [c, x_1] [x_2, x_3, x_4] - [c, x_1, x_5] [x_2, x_3, x_4].
\end{align*}
Note that $ [c, x_1] [x_2, x_3, x_4]$ is an element of the form (\ref{generatorTns[xxx]});  it follows that $ [c, x_1] [x_2, x_3, x_4] \in I^{(n)}$, therefore 
\[
[c, x_1] [x_2, x_3, x_4] x_5 - x_5 [c, x_1] [x_2, x_3, x_4] \in I^{(n)}.
\]
On the other hand,  by Proposition \ref{proposition3vinI'}, the element $[c, x_1, x_5] [x_2, x_3, x_4] $  belong to the two-sided ideal of $A$  generated by the elements of the forms $[c, x_1, z_1, z_2]$ and $[c, x_1, z_1] [z_2, z_3] + [c, x_1, z_2][z_1, z_3]$ $(z_i \in X)$. Since the latter elements belong to $I^{(n)}$, we have
\[
[c, x_1, x_5] [x_2, x_3, x_4] \in I^{(n)} .
\]
It follows that $[c, x_1] [x_2, x_3, x_4, x_5] \in I^{(n)}$, as required.
\end{proof}

Since  $\bigl[ [x_2, x_3], [x_4, x_5] \bigr] = [x_2, x_3, x_4, x_5] - [x_2, x_3, x_5, x_4]$, we have 

%%%%%%%%%%%%%%%%%%%%%%%%%%%
\begin{corollary}
\label{corollary[cx][[xx][xx]]}
For all $x_i \in X$, we have $[c, x_1] \bigl[ [x_2, x_3], [x_4, x_5] \bigr] \in I^{(n)}$.
\end{corollary}

%%%%%%%%%%%%%%%%%%%%%%%%%%%
\begin{lemma}
\label{lemma[cx]g(1234)}
For all $x, z_i \in X$,
\[
[c, x] g(z_1, z_2, z_3, z_4) \equiv [c, x] g(z_4, z_2, z_3, z_1) \pmod{I^{(n)} } ,
\]
\[
[c, x] g(z_1, z_2, z_3, z_4)  \equiv [c, x] g(z_2, z_1, z_4, z_3) \pmod{I^{(n)} } .
\]
\end{lemma}

%%%%%%%%%%%%%%%%%%%%%%%%%%%
\begin{proof}
We have
\begin{align*}
& g(z_1, z_2, z_3, z_4) = [z_1, z_2][z_3, z_4] + [z_1, z_3][z_2, z_4]  = [z_4, z_3] [z_2, z_1] + \bigl[ [z_1, z_2], [z_3, z_4] \bigr] + [z_4, z_2] [z_3, z_1]
\\
& +  \bigl[ [z_1, z_3] , [z_2, z_4] \bigr] = g(z_4, z_2, z_3, z_1) + \bigl[ [z_1, z_2], [z_3, z_4] \bigr] + \bigl[ [z_1, z_3] , [z_2, z_4] \bigr]
\end{align*}
so, by Corollary \ref{corollary[cx][[xx][xx]]},
\[
[c, x] g(z_1, z_2, z_3, z_4) \equiv [c, x] g(z_4, z_2, z_3, z_1) \pmod{I^{(n)} } .
\]
On the other hand,
\begin{align*}
& g(z_1, z_2, z_3, z_4) = [z_1, z_2][z_3, z_4] + [z_1, z_3][z_2, z_4]  = [z_2, z_1] [z_4, z_3] 
\\
& + [z_2, z_4] [z_1, z_3] - \bigl[ [z_2, z_4], [z_1, z_3] \bigr] = g(z_2, z_1,z_4, z_3) - \bigl[ [z_2, z_4], [z_1, z_3] \bigr] 
\end{align*}
so, by Corollary \ref{corollary[cx][[xx][xx]]},
\[
[c, x] g(z_1, z_2, z_3, z_4)  \equiv [c, x] g(z_2, z_1, z_4, z_3) \pmod{I^{(n)} } .
\]
This completes the proof of Lemma \ref{lemma[cx]g(1234)}.
\end{proof}

%%%%%%%%%%%%%%%%%%%%%%%%%%%
Now we are in a position to check that the elements of the form (\ref{generatorTns[xx]xx]}) belong to $I^{(n)}$. We have 
\begin{align*}
& g( x_2, x_3, x_4, x_5) + g( x_2, x_3, x_5, x_4)  + g(x_2, x_4, x_5, x_3)  = \bigl( [x_2, x_3] [x_4, x_5] + [x_2, x_4] [x_3, x_5] \bigr) 
\\
& + \  \bigl( [x_2, x_3] [x_5, x_4] + [x_2, x_5] [x_3, x_4] \bigr) + \bigl( [x_2, x_4] [x_5, x_3] + [x_2, x_5] [x_4, x_3] \bigr) = 0
\end{align*}
so
\begin{equation}
\label{[cx](g+g+g)=0}
[c, x_1] \bigl( g( x_2, x_3, x_4, x_5) + g( x_2, x_3, x_5, x_4)  + g(x_2, x_4, x_5, x_3) \bigr) = 0.
\end{equation}
On the other hand,  by Lemmas \ref{lemma[c1]g(2345)} and \ref{lemma[cx]g(1234)},
\begin{align*}
& [c, x_1] g(x_2, x_3, x_5, x_4) \equiv [c, x_1] g(x_4, x_3, x_5, x_2) \pmod{I^{(n)} } 
\\
\equiv \ & -  [c, x_4] g(x_1, x_3, x_5, x_2) \pmod{I^{(n)} } \equiv \  - [c, x_4] g(x_3, x_1, x_2, x_5) \pmod{I^{(n)} }
\\
\equiv \  & [c, x_3] g(x_4, x_1, x_2, x_5) \pmod{I^{(n)} }  \equiv \   [c, x_3] g(x_1, x_4, x_5, x_2) \pmod{I^{(n)} } 
 \\
\equiv \ & - [c, x_1] g(x_3, x_4, x_5, x_2) \pmod{I^{(n)} }   \equiv \  - [c, x_1] g(x_2, x_4, x_5, x_3) \pmod{I^{(n)} } ,
\end{align*}
that is, 
\begin{equation}
\label{[cx](g+g)=0}
[c, x_1] \bigl( g( x_2, x_3, x_5, x_4)  + g(x_2, x_4, x_5, x_3) \bigr) \equiv 0 \pmod{I^{(n)} }.
\end{equation}
It follows from (\ref{[cx](g+g+g)=0}) and (\ref{[cx](g+g)=0}) that 
\begin{equation*}
[c, x_1] g( x_2, x_3, x_4, x_5) \equiv 0 \pmod{I^{(n)} }.
\end{equation*}

Recall that each element of the form (\ref{generatorTns[xx]xx]}) can be written as $[c, x_1] g( x_2, x_3, x_4, x_5) $ where $x_i \in X$, $c = [y_1, y_2, \dots , y_{n-3}]$, $y_1 \in X$, $y_2, \dots , y_{n-3}\in X \cup X^2$. Thus, each element of the form (\ref{generatorTns[xx]xx]}) belongs to $I^{(n)}$, as required. 

The proof of Theorem \ref{generatorsTn} is completed.

%%%%%%%%%%%%%%%%%%%%%%%%%%%%%%%%%%%%%%%%
%%%%%%%%%%%%%%%%%%%%%%%%%%%%%%%%%%%%%%%%


\begin{thebibliography}{99}

\bibitem{AE16}
Nabilah Abughazalah, Pavel Etingof, \textit{On properties of the lower central series of associative algebras}, Journal of Algebra and its Applications \textbf{15} (2016), 1650187 (24 pages). DOI: 10.1142/S0219498816501875. arXiv:1508.00943 [math.RA].

%
%\bibitem{AK09} 
%Elena~V.~Aladova and Alexei~N.~Krasilnikov, \emph{Polynomial identities in nil-algebras}, Transactions of the American Mathematical Society \textbf{361} (2009), no. 11, %5629--5646.
%

\bibitem{AJ10}
Noah Arbesfeld, David Jordan, \textit{New results on the lower central series quotients of a free associative algebra}, Journal of Algebra \textbf{323} (2010) 1813--1825. DOI: 10.1016/j.jalgebra.2009.12.024. arXiv:0902.4899 [math.RA].

\bibitem{BB11}
Martina Balagovi\'c, Anirudha Balasubramanian, \textit{On the lower central series quotients of a graded associative algebra}, Journal of Algebra \textbf{328} (2011) 287--300. DOI: 10.1016/j.jalgebra.2010.08.023. arXiv:1004.3735 [math.RA].

\bibitem{BJ10}
Asilata Bapat, David Jordan, \textit{Lower central series of free algebras in symmetric tensor categories}, Journal of Algebra \textbf{373} (2013) 299--311. DOI: 10.1016/j.jalgebra.2012.10.001. arXiv:1001.1375 [math.RA].

%
%\bibitem{BOR11}
%C. Bekh-Ochir, S.A. Rankin, \textit{Examples of associative algebras for which the $T$-space of central polynomials is not finitely based}, Israel J. Math. \textbf{186} (2011), %333--347.
%

%
%\bibitem{BOR16}
%Chuluundorj Bekh-Ochir, David Riley, An explicit basis for the Grassmann $T$-space, Journal of Algebra \textbf{466} (2016), 63--72. DOI: 10.1016/j.jalgebra.2016.07.017
%

%
%\bibitem{Bel99}
%A.Ya. Belov, \textit{On non-Specht varieties} (Russian), Fundam. Prikl. Mat. \textbf{5} (1999), 47--66.
%

%\bibitem{Bel00}
%A.Ya. Belov, \textit{Counterexamples to the Specht problem} (Russian). Mat. Sb. \textbf{191} (2000), 13--24. English translation in Sb. Math. \textbf{191} (2000), 329--340.
%

\bibitem{BEJKL12}
Surya Bhupatiraju,  Pavel Etingof, David Jordan,  William Kuszmaul,  Jason Li, \textit{Lower central series of a free associative algebra over the integers and finite fields}, Journal of Algebra \textbf{372} (2012), 251--274. DOI: 10.1016/j.jalgebra.2012.07.052. arXiv:1203.1893 [math.RA].

%
%\bibitem{BKKS10}
%Ant\^onio Pereira Brand\~ao Jr., Plamen Koshlukov, Alexei Krasilnikov and \'Elida Alves da Silva, \emph{The central polynomials for the Grassmann algebra,} Israel Journal of %Mathematics \textbf{179} (2010), 127--144.
%

\bibitem{CFZ15}
Katherine Cordwell, Teng Fei, Kathleen Zhou, \textit{On lower central series quotients of finitely generated algebras over $\mathbb Z$}, Journal of Algebra \textbf{423} (2015) 559--572. DOI: 10.1016/j.jalgebra.2014.10.018. arXiv:1309.1237 [math.RA].

\bibitem{dCKras13}
Eudes A. da Costa, Alexei Krasilnikov, \textit{Relations in universal Lie nilpotent associative algebras of class $4$}, to appear in Communications in Algebra. DOI: 10.1080/00927872.2017.1347661. arXiv:1306.4294 [math.RA].

\bibitem{Dangovski15}
Rumen R. Dangovski, On the maximal containments of lower central series ideals, arXiv 1509:08030 [math.RA].

\bibitem{DK15}
Galina Deryabina, Alexei Krasilnikov, \textit{The torsion subgroup of the additive group of a Lie nilpotent associative ring of class 3}, Journal of Algebra \textbf{428} (2015), 230--255. DOI: 10.1016/j.jalgebra.2015.01.009. arXiv:1308.4172 [math.RA].

%
%\bibitem{DK15-2} 
%Galina Deryabina and Alexei Krasilnikov, \textit{The subalgebra of graded central polynomials of an associative algebra,} Journal of Algebra \textbf{425} (2015), 313--323.
%

\bibitem{DK17}
Galina Deryabina, Alexei  Krasilnikov, \textit{Products of commutators in a Lie nilpotent associative algebra}, Journal of Algebra \textbf{469} (2017), 84--95. DOI: 10.1016/j.jalgebra.2016.08.031. 	arXiv:1509.08890 [math.RA].

\bibitem{DE08}
Galyna Dobrovolska, Pavel Etingof, \textit{An upper bound for the lower central series quotients of a free associative algebra}, International Mathematics Research Notices \textbf{2008} no. 12, Art. ID rnn039, 10 pp. DOI: 10.1093/imrn/rnn039. arXiv:0801.1997 [math.RA].

\bibitem{DKM08}
Galyna Dobrovolska, John Kim, Xiaoguang Ma, \textit{On the lower central series of an associative algebra (with an appendix by Pavel Etingof)}, Journal of Algebra \textbf{320} (2008) 213--237. DOI: 10.1016/j.jalgebra.2008.02.019. arXiv:0709.1905 [math.QA].

\bibitem{EKM09}
Pavel Etingof, John Kim,  Xiaoguang Ma, \textit{On universal Lie nilpotent associative algebras}, Journal of Algebra \textbf{321} (2009), 697--703. DOI: 10.1016/j.jalgebra.2008.09.042. arXiv:0805.1909 [math.RA].

\bibitem{FS07}
Boris Feigin, Boris Shoikhet, \textit{On $[A, A]/[A, A, A]$ and on a $W_n$-action on the consecutive commutators of free associative algebras}, Mathematical Research Letters \textbf{14} (2007), 781--795. DOI: 10.4310/MRL.2007.v14.n5.a7. arXiv:math/0610410 [math.QA].

\bibitem{GK01}
A. Giambruno, P. Koshlukov, \textit{On the identities of the Grassmann algebra in characteristic $p>0$}. Israel Journal of Mathematics \textbf{122} (2001), 305--316. 
DOI: 10.1007/BF02809905.


%
%\bibitem{GKS12} 
%Dimas Jos\'e Gon\c{c}alves, Alexei Krasilnikov and Irina Sviridova, \textit{Limit $T$-subspaces and the central polynomials on $n$ variables of the Grassmann algebra,} Journal of %Algebra \textbf{371} (2012), 156--174.
%

%
%\bibitem{GKS14}
%Dimas Jos\'e Gon\c{c}alves, Alexei Krasilnikov and Irina Sviridova, \textit{Limit $T$-subalgebras in free associative algebras}, Journal of Algebra \textbf{412} (2014), 264--280.
%

\bibitem{Gordienko07}
A.S. Gordienko, \textit{Codimensions of commutators of length $4$}, Russian Mathematical Surveys \textbf{62} (2007), 187--188. DOI: 10.1070/RM2007v062n01ABEH004383.

%
%\bibitem{Grishin99}
%A.V. Grishin, \textit{Examples of $T$-spaces and $T$-ideals of characteristic 2 without the finite basis property}. (Russian) Fundam. Prikl. Mat.  \textbf{5}  (1999),  101--118.
%

%
%\bibitem{Grishin00}
%A.V. Grishin, \textit{On non-Spechtianness of the variety of associative rings that satisfy the identity $x\sp {32}=0$}, Electron. Res. Announc. Amer. Math. Soc.  \textbf{6}  (2000),
%50--51 (electronic).
%

%
%\bibitem{GTs09} 
%A.V.Grishin, L.M.Tsybulya,\textit{On the $T$-space and multiplicative structure of a relatively free Grassmann algebra}, Sb. Math.  \textbf{200} (2009), no. 9-10, 1299--1338.
%

%
%\bibitem{Grishin10}
%A.V. Grishin, \textit{On the structure of the centre of a relatively free Grassmann algebra} (Russian), Uspekhi Mat. Nauk \textbf{65} (2010), no. 4, 191--192. English translation in %Russ. Math. Surv. \textbf{65} (2010), 781--782.
%

\bibitem{GP15}
A.V. Grishin, S.V. Pchelintsev, \textit{On centres of relatively free associative algebras with a Lie nilpotency identity}, Sbornik: Mathematics \textbf{206} (2015), 1610--1627.   DOI: 10.1070/SM2015v206n11ABEH004506.

\bibitem{GTS11}
A.V. Grishin,  L.M. Tsybulya,  A.A. Shokola, \textit{On  $T$-spaces and relations in relatively free, Lie nilpotent, associative algebras}, Journal of Mathematical Sciences \textbf{177} (2011), 868--877. DOI: 10.1007/s10958-011-0515-3.

\bibitem{GK99}
C.K. Gupta, A.N. Krasil'nikov, \textit{A solution of a problem of Plotkin and Vovsi and an application to varieties of groups}, Journal of the Australian Mathematical Society (Series A) \textbf{67} (1999) , 329--355. DOI: 10.1017/S1446788700002056.


%
%\bibitem{GK02}
% C.K.~Gupta and A.N.~Krasilnikov, \textit{A non-finitely based system of polynomial identities which contains the identity $x^6 = 0$}, Quarterly Journal of Mathematics 
%\textbf{53} (2002), no. 2, 173--183.
%

%
%\bibitem{GK02-2}
% C.K.~Gupta and A.N.~Krasilnikov, {\em A simple example of a non-finitely based system of polynomial identities\/}, Communications in Algebra {\bf 30} (2002), no. 10, 4851--4866.
%

\bibitem{Huppert67}
B. Huppert, Endliche Gruppen I, Berlin-Heidelberg-New York: Springer, 1967, 793 p.

\bibitem{Jennings47}
S.A.	Jennings, \textit{On rings whose associated Lie rings are nilpotent}, Bulletin of the American Mathematical Society \textbf{53} (1947), 593--597. DOI: 10.1090/S0002-9904-1947-08844-3.

\bibitem{JO15}
D. Jordan, H. Orem, \textit{An algebro-geometric construction of lower central series of associative algebras}, International Mathematics Research Notices \textbf{2015}, no. 15, 6330--6352. DOI: 10.1093/imrn/rnu125. arXiv:1302.3992 [math.AG].

\bibitem{K-BKR16}
A. Kanel-Belov, Ya. Karasik, L.H. Rowen, Computational Aspects of Polynomial Identities: Volume l, Kemer's Theorems. Boca Raton-London-New York: CRC Press, 2016, 408 p.

\bibitem{Khukhro97}
E.I. Khukhro, $p$-Automorphisms of finite $p$-groups, London Mathematical Society Lecture Notes, \textbf{246}. Cambridge-New Your-Melbourne: Cambridge University Press, 1997, 204 p. 

\bibitem{Kras13}
Alexei Krasilnikov, \textit{The additive group of a Lie nilpotent associative ring}, Journal of Algebra \textbf{392} (2013),  10--22. DOI: 10.1016/j.jalgebra.2013.06.021. arXiv:1204.2674 [math.RA].

\bibitem{Latyshev63}
V.N. Latyshev, \textit{On the choise of basis in a $T$-ideal}, Sibirskii Matematicheskii Zhurnal (Siberian Mathematical Journal) \textbf{4} (1963), 1122--1127. (in Russian)

%
%\bibitem{Regev16}
%Amitai Regev, Growth for the central polynomials, Communications in Algebra, \textbf{44} (2016) 4411--4421.
%

%
%\bibitem{Shchigolev99}
%V. V. Shchigolev, \textit{Examples of infinitely based $T$-ideals} (Russian), Fundam. Prikl. Mat. \textbf{5} (1999), 307--312.
%

%
%\bibitem{Shchigolev00}
%V.V. Shchigolev, \textit{Examples of infinitely basable $T$-spaces}. (Russian)  Mat. Sb. \textbf{191} (2000), no. 3, 143--160; English translation in Sb. Math.  \textbf{191} (2000), %459--476.
%

\bibitem{SS90}
R.K. Sharma, J.B. Srivastava, \textit{Lie ideals in group rings}, Journal of Pure and Applied Algebra \textbf{63} (1990), 67--80. DOI: 10.1016/0022-4049(90)90056-N

\bibitem{Sysak10}
Ya.P. Sysak, \textit{The adjoint group of radical rings and related questions}, In: Ischia group theory 2010, 344--365, World Scientific Publ., NJ, 2012.

\bibitem{Volichenko78}
I.B. Volichenko, \textit{The T-ideal generated by the element $[x_1, x_2, x_3, x_4]$},  Minsk: Institute of Mathematics of the Academy os Sciences of the Belorussian SSR, Preprint 22, 1978, 13 p. (in Russian)


%%%%%%%%%%%%%%%%%%%%%%%%%%%%%%%%%%%%%%%%

%\bibitem{AS01}
%Amberg, B., Sysak, Ya.  Associative rings whose adjoint semigroup is locally nilpotent, \textit{Arch. Math. (Basel)} 76  (2001)  426--435.


%
%\bibitem{LS85}
%F. Levin, S. Sehgal, \textit{On Lie nilpotent group rings},  Journal of Pure and Applied Algebra \textbf{37} (1985), 33--39. DOI: 10.1016/0022-4049(85)90085-4
%




\end{thebibliography}
\end{document}